\documentclass[a4, 12pt]{amsart}
\usepackage[mathscr]{eucal}
\usepackage{amssymb}
\usepackage{latexsym}
\usepackage{amsthm}
\theoremstyle{plain}
\newtheorem{theorem}{Theorem}[section]
\newtheorem{proposition}{Proposition}[section]
\newtheorem{corollary}{Corollary}[section]
\newtheorem{remark}{Remark}[section]
\newtheorem{lemma}{Lemma}[section]
\newtheorem{example}{Example}[section]
\newtheorem{definition}{Definition}[section]

\makeatletter
\@addtoreset{equation}{section}

\title[Complete $\lambda$-hypersurfaces ]
{Complete $\lambda$-hypersurfaces of  weighted volume-preserving mean curvature flow }
\author{Qing-Ming Cheng and  Guoxin Wei}
\address{Qing-Ming Cheng \\ Department of Applied Mathematics, Faculty of Sciences ,
Fukuoka  University, 814-0180, Fukuoka,  Japan, cheng@fukuoka-u.ac.jp}
\address{Guoxin Wei \\  School of Mathematical Sciences, South China Normal University,
510631, Guangzhou,  China, weiguoxin@tsinghua.org.cn}

\begin{document}
\maketitle

\begin{abstract}
\noindent In this paper,  we introduce a  definition of  $\lambda$-hypersurfaces of weighted
volume-preserving mean curvature flow in Euclidean space. We prove  that $\lambda$-hypersurfaces
are critical points of the weighted area functional for the weighted volume-preserving variations.
Furthermore, we  classify  complete $\lambda$-hypersurfaces with polynomial area growth and
$H-\lambda\geq 0$, which  are  generalizations of the results due to Huisken \cite{[H2]}, Colding-Minicozzi \cite{[CM]}.
We also define a $\mathcal{F}$-functional and study $\mathcal{F}$-stability of $\lambda$-hypersurfaces, which
extend a result of  Colding-Minicozzi \cite{[CM]}. Lower bound growth and upper bound growth of the area for
complete and non-compact $\lambda$-hypersurfaces are also studied.

\end{abstract}

\footnotetext{ 2001 \textit{ Mathematics Subject Classification}: 53C44, 53C42.}

\footnotetext{{\it Key words and phrases}: the weighted volume-preserving mean curvature flow, the weighted area functional,
$\mathcal{F}$-stability, weak stability, $\lambda$-hypersurfaces.}

\footnotetext{The first author was partially  supported by JSPS Grant-in-Aid for Scientific Research (B): No. 24340013
and Challenging Exploratory Research No. 25610016.
The second author was partly supported by grant No. 11371150 of NSFC.}

\section {Introduction}

\noindent
Let $X: M\rightarrow \mathbb{R}^{n+1}$ be a smooth $n$-dimensional immersed hypersurface in the $(n+1)$-dimensional
Euclidean space $\mathbb{R}^{n+1}$.
A  family $X(\cdot, t)$ of smooth immersions:
$$
X(\cdot, t):M\to  \mathbb{R}^{n+1}
$$
with  $X(\cdot, 0)=X(\cdot)$ is called  a mean curvature flow
if  they satisfy
\begin{equation*}
\dfrac{\partial X(p,t)}{\partial t}=\mathbf{H}(p,t),
\end{equation*}
where $\mathbf{H}(t)=\mathbf{H}(p,t)$ denotes the mean curvature vector  of hypersurface  $M_t=X(M^n,t)$ at point $X(p,t)$.
Huisken \cite{[H1]} proved that the mean curvature flow $M_t$ remains smooth and convex
until it becomes extinct at a point in the finite time.   If we rescale the flow about the point, the rescaling
converges to the round sphere. An immersed hypersurface $X:M\to  \mathbb{R}^{n+1}$ is called {\it a self-shrinker}  if
\begin{equation*}
H+\langle X,N\rangle=0,
\end{equation*}
where $H$ and $N$ denote the mean curvature and the  unit normal vector of $X:M\to  \mathbb{R}^{n+1}$, respectively.
 $\langle \cdot , \cdot\rangle $ denotes the standard inner product in $\mathbb{R}^{n+1}$.
It is  known that self-shrinkers play an important role in the study of the mean curvature flow
because they describe all possible blow ups at a given singularity of the mean curvature flow.

\noindent
For $n=1$,  Abresch and Langer \cite{[AL]}  classified all smooth closed self-shrinker curves
in $\mathbb{R}^2$ and showed that the  round circle is the only embedded self-shrinker.
For $n\geq 2$, Huisken \cite{[H2]}  studied compact self-shrinkers.
He proved that if  $M$ is an $n$-dimensional compact self-shrinker
with non-negative mean curvature  in $\mathbb{R}^{n+1}$, then $X(M)=S^n(\sqrt{n})$.
In the remarkable paper  \cite{[CM]}, Colding and Minicozzi have classified complete self-shrinkers with non-negative mean
curvature  and polynomial area  growth (which is called polynomial volume growth in \cite{[CM]} and \cite{[H3]})
in $\mathbb{R}^{n+1}$.  We should remark that Huisken \cite{[H3]} proved the  same results if
the squared norm of the second fundamental form is bounded.  Colding and Minicozzi \cite{[CM]}
have introduced a notation of $\mathcal{F}$-functional
and computed the first and the second variation formulas
of the $\mathcal{F}$-functional. They have proved that  an immersed hypersurface $X:M\to  \mathbb{R}^{n+1}$ is
a self-shrinker if and only if it is a critical point of the $\mathcal{F}$-functional.  Furthermore,
they have given a complete classification of the $\mathcal{F}$-stable complete self-shrinkers
with polynomial area growth.

\noindent
On the other hand, Huisken \cite{[H4]} studied the volume-preserving mean curvature flow
\begin{equation*}
\dfrac{\partial X(t)}{\partial t}=-h(t) N(t) +\mathbf{H}(t),
\end{equation*}
where $X(t)=X(\cdot, t)$, $h(t)=\frac{\int_MH(t)d\mu_t}{\int_Md\mu_t}$  and $N(t)$ is the unit normal vector of $X(t):M\to  \mathbb{R}^{n+1}$.
He proved that if the initial hypersurface is uniformly convex, then the above volume-preserving mean curvature flow has a smooth
solution and it converges to a round sphere. Furthermore,  by making use of the Minkowski formulas, Guan and Li \cite{[GL]}
have studied the following type of mean curvature flow
\begin{equation*}
\dfrac{\partial X(t)}{\partial t}=-n N(t) +\mathbf{H}(t),
\end{equation*}
which is also a volume-preserving mean curvature flow. They have gotten that the flow converges to a solution of the isoperimetric problem
if the initial hypersurface is a smooth compact, star-shaped hypersurface.

\noindent
In this paper, we consider a new type of mean curvature flow:
\begin{equation}
\dfrac{\partial X(t)}{\partial t}=-\alpha(t) N(t) +\mathbf{H}(t),
\end{equation}
with
$$
\alpha(t) =\dfrac{\int_MH(t)\langle N(t), N\rangle e^{-\frac{|X|^2}2}d\mu}{\int_M\langle N(t), N\rangle e^{-\frac{|X|^2}2}d\mu},
$$
where  $N$ is the unit normal vector of $X:M\to  \mathbb{R}^{n+1}$.
We define {\it a weighted volume} of $M_t$ (see, section 2) by
$$
V(t)=\int_M\langle X(t), N\rangle e^{-\frac{|X|^2}{2}}d\mu.
$$
We can prove that the flow (1.1) preserves the weighted volume $V(t)$. Hence, we call the flow (1.1)  {\it a weighted  volume-preserving mean curvature flow}.

\noindent
The properties of solutions  of the  weighted  volume-preserving mean curvature flow (1.1)
will be studied in Cheng and Wei \cite{[CW1]}.

\noindent
This paper is organized as follows. In section 2,  we give a  definition of the weighted volume
and the first variation formula of the weighted area functional for all weighted volume-preserving
variations is given. As critical points of it, $\lambda$-hypersurface is defined. Self-similar solutions
of the weighted volume-preserving mean curvature  flow is considered. In section 3, the basic
properties of  $\lambda$-hypersurfaces are studied. In section 4, we give a classification for
compact  $\lambda$-hypersurfaces with $H-\lambda\geq 0$.  In sections 5 and 6, we define
$\mathcal F$-functional. The first and second variation formulas of $\mathcal F$-functional
are proved.  Notation of $\mathcal F$-stability and $\mathcal F$-unstability  of $\lambda$-hypersurfaces are
introduced. We prove that
spheres $S^n(r)$ with $r\leq \sqrt n$ or $r>\sqrt {n+1}$ are  $\mathcal F$-stable and
spheres $S^n(r)$ with $\sqrt n<r\leq \sqrt {n+1}$ are  $\mathcal F$-unstable.
In section 7, we study the  weak stability of  the weighted area functional for the weighted
volume-preserving variations. In section 8, a classification for
complete and non-compact   $\lambda$-hypersurfaces with polynomial area growth and
$H-\lambda\geq 0$ is given.  In sections 9 and 10, the area growth
of complete and non-compact $\lambda$-hypersurfaces are studied.
\vskip 2mm
\noindent
{\bf Acknowledgement}. A part of this work was finished when the first author
visited to Beijing Normal University. We would like to express our gratitude to
Professor Tang Zizhou and Dr. Yan Wenjiao for warm hospitality.

\vskip 5mm

\section{The first variation formula and $\lambda$-hypersurfaces}
\noindent
Let $X: M^n\rightarrow\mathbb{R}^{n+1}$ be an
$n$-dimensional connected hypersurface of the $(n+1)$-dimensional Euclidean space
$\mathbb{R}^{n+1}$. We choose a local orthonormal frame field
$\{e_A\}_{A=1}^{n+1}$ in $\mathbb{R}^{n+1}$ with dual coframe field
$\{\omega_A\}_{A=1}^{n+1}$, such that, restricted to $M^n$,
$e_1,\cdots, e_n$ are tangent to $M^n$.
Then we have
\begin{equation*}
dX=\sum_i\limits \omega_i e_i, \quad de_i=\sum_j\limits \omega_{ij}e_j+\omega_{i n+1}e_{n+1}
\end{equation*}
and
\begin{equation*}
de_{n+1}=\sum_i\limits\omega_{n+1 i}e_i.
\end{equation*}
We restrict these forms to $M^n$, then
$$
\omega_{n+1}=0,\ \ \omega_{n+1i}=-\sum_{j=1}^nh_{ij}\omega_j,\ \ h_{ij}=h_{ji},
$$
where $h_{ij}$ denotes the component of the second fundamental form of $X: M^n\rightarrow\mathbb{R}^{n+1}$.
$\mathbf{H}=\sum_{j=1}^nh_{jj}e_{n+1}$ is the mean curvature vector field, $H=|\mathbf{H}|=\sum_{j=1}^nh_{jj}$
is the mean curvature and $II=\sum_{i,j}h_{ij}\omega_i\otimes\omega_je_{n+1}$ is the second fundamental form
of $X: M^n\rightarrow\mathbb{R}^{n+1}$.
Let
$$
f_{,i}=\nabla_i f,  \  f_{,ij}=\nabla_j\nabla_if, \  \ h_{ijk}=\nabla_kh_{ij} \  \ {\rm and} \ \ h_{ijkl}=\nabla_l\nabla_kh_{ij},
$$
where
$\nabla_j $ is the covariant differentiation operator.The Gauss equations and Codazzi equations are given by
\begin{equation}\label{eq:12-6-4}
R_{ijkl}=h_{ik}h_{jl}-h_{il}h_{jk},
\end{equation}
\begin{equation}\label{eq:12-6-5}
h_{ijk}=h_{ikj},
\end{equation}
where $R_{ijkl}$ and $h_{ijk}$ denote  components of curvature tensor and components of  the covariant derivative of $h_{ij}$,
respectively.  Furthermore,  we have the Ricci formula:
\begin{equation}\label{eq:12-6-6}
h_{ijkl}-h_{ijlk}=\sum_{m=1}^nh_{im}R_{mjkl}+\sum_{m=1}^nh_{mj}R_{mikl}.
\end{equation}
For a constant vector $a\in \mathbb{R}^{n+1}$, one has
$$
\langle X,a\rangle _{,i}=\langle e_i,a\rangle ,\ \ \langle N,a\rangle _{,i}=-\sum_jh_{ij}\langle e_j,a\rangle,
$$

$$
\langle X,a\rangle _{,ij}=h_{ij}\langle N,a\rangle ,
$$

$$
 \langle N,a\rangle _{,ij}=-\sum_kh_{ijk}\langle e_k,a\rangle -\sum_kh_{ik}h_{jk}\langle N,a\rangle.
$$

\noindent
We call $X(t)$ is a variation of $X$ if $X(t):M\rightarrow \mathbb{R}^{n+1}$, $t\in (-\varepsilon,\varepsilon)$ is a family of  immersions
with  $X(0)=X$. For $X_0\in \mathbb{R}^{n+1}$ and a real number $t_0 $, we define {\it a weighted area function}
$A:(-\varepsilon,\varepsilon)\rightarrow\mathbb{R}$ by
$$
A(t)=\int_Me^{-\frac{|X(t)-X_0|^2}{2t_0}}d\mu_t,
$$
where $d\mu_t$ is the area element of $M$ in the metric induced by $X(t)$.
The {\it weighted volume function} $V:(-\varepsilon,\varepsilon)\rightarrow \mathbb{R}$ is defined by
$$
V(t)=\int_M\langle X(t)-X_0,N\rangle e^{-\frac{|X-X_0|^2}{2t_0}}d\mu.
$$
In this paper, we only consider compactly supported variations. By a direct calculations, we have the following first variation formulas of $A(t)$ and $V(t)$:

\begin{lemma}\label{lemma 100}
\begin{equation}
\frac{d A(t)}{dt}=\int_M \biggl(-\frac{\langle X(t)-X_0,\frac{\partial X(t)}{\partial t}\rangle }{t_0}
-H(t)\langle \frac{\partial X(t)}{\partial t}, N(t)\rangle \biggl)e^{-\frac{|X(t)-X_0|^2}{2t_0}}d\mu_t,
\end{equation}

\begin{equation}
\frac{d V(t)}{dt}=\int_M \langle \frac{\partial X(t)}{\partial t},N\rangle e^{-\frac{|X-X_0|^2}{2t_0}}d\mu.
\end{equation}
\end{lemma}
\noindent
Let $\frac{\partial X(t)}{\partial t}=W(t)$. Then the vector field $\frac{\partial X(t)}{\partial t}|_{t=0}=W(0)=W$ is called
{\it a variation vector field}. Set $f(t)=\langle W(t),N(t)\rangle $, where $N(t)$ is the normal vector of $M_t$, $N(0)=N$.
In this paper, we only consider the normal variation vector field, which can be expressed as $\frac{\partial X(t)}{\partial t}|_{t=0}=fN$.
We say a variation of $X$ is  {\it a weighted volume-preserving variation } if $V(t)=V(0)$ for all $t$, that is
\begin{equation}
\aligned
0&=\frac{d V(t)}{dt}=\int_M \langle \frac{\partial X(t)}{\partial t},N\rangle e^{-\frac{|X-X_0|^2}{2t_0}}d\mu\\
 &=\int_Mf(t)\langle N(t),N\rangle e^{-\frac{|X-X_0|^2}{2t_0}}d\mu.
\endaligned
\end{equation}
We can prove the following lemma using the same method as that of the lemma 2.4 of \cite{[BC]}.

\begin{lemma}\label{lemma 101}
Given a smooth function $f:M\rightarrow \mathbb{R}$ with $\int_M f e^{-\frac{|X-X_0|^2}{2t_0}}d\mu=0$,
there exists a weighted volume-preserving normal variation such that  its variation vector field is $fN$.
\end{lemma}

\noindent
Let
$$
\lambda=\frac{1}{A}\int_M(\langle \frac{X-X_0}{t_0},N\rangle +H)e^{-\frac{|X-X_0|^2}{2t_0}}d\mu,
$$
with
$$
A=\int_Me^{-\frac{|X-X_0|^2}{2t_0}}d\mu
$$
and define
$J:(-\varepsilon,\varepsilon)\rightarrow\mathbb{R}$ by
$$
J(t)=A(t)+\lambda V(t),
$$
for constant $\lambda$.
Then, one has
\begin{proposition}
Let $X:M\rightarrow \mathbb{R}^{n+1}$ be an immersion. The following statements are equivalent with each other:

\begin{enumerate}
\item $\langle \frac{X-X_0}{t_0},N\rangle +H=\lambda$.
\item  For all weighted volume-preserving variations, $A^{'}(0)=0$.
\item  For all arbitrary variations, $J^{'}(0)=0$.
\end{enumerate}
\end{proposition}

\begin{proof} From Lemma \ref{lemma 100}, we have (1)$\Rightarrow$ (3) and (3)$\Rightarrow$ (2).
We next prove (2)$\Rightarrow$ (1).
Assume that at a point $p\in M$, we have $(\langle \frac{X-X_0}{t_0},N\rangle +H-\lambda)(p)\neq 0$.
We can assume that
$(\langle \frac{X-X_0}{t_0},N\rangle +H-\lambda)(p)> 0$. Let
$$
M^{+}=\{q\in M: (\langle \frac{X-X_0}{t_0},N\rangle +H-\lambda)(p)> 0\},
$$
$$
M^{-}=\{q\in M: (\langle \frac{X-X_0}{t_0},N\rangle +H-\lambda)(p)< 0\}.
$$
Let $\varphi$ and $\psi$ be non-negative real  smooth functions on $M$ such that
$$
p\in {\rm supp}\varphi\subset M^{+}, \ \ \ {\rm supp}\psi\subset M^{-},
$$
and
$$
\int_M(\varphi+\psi)(\langle \frac{X-X_0}{t_0},N\rangle +H-\lambda) e^{-\frac{|X-X_0|^2}{2t_0}}d\mu=0.
$$
Since $\int_M(\langle \frac{X-X_0}{t_0},N\rangle +H-\lambda) e^{-\frac{|X-X_0|^2}{2t_0}}d\mu=0$, we know that such a choice is possible.
Let $f=(\varphi+\psi)(\langle \frac{X-X_0}{t_0},N\rangle +H-\lambda)$, then $\int_M fe^{-\frac{|X-X_0|^2}{2t_0}}d\mu=0$.
By Lemma \ref{lemma 101}, we get a weighted volume-preserving variation such that its variation vector field is $fN$.
From our  assumption,
$$
A^{'}(0)=\int_M (-\frac{\langle X-X_0,N\rangle }{t_0}-H)fe^{-\frac{|X-X_0|^2}{2t_0}}d\mu=0.
$$
Hence, we have
\begin{equation}
\aligned
0&=\int_Mf(\langle \frac{X-X_0}{t_0},N\rangle +H-\lambda)e^{-\frac{|X-X_0|^2}{2t_0}}d\mu\\
 &=\int_M(\varphi+\psi)(\langle \frac{X-X_0}{t_0},N\rangle +H-\lambda)^2e^{-\frac{|X-X_0|^2}{2t_0}}d\mu\\
 &> 0.
\endaligned
\end{equation}
It is a contradiction. It follows that $\langle \frac{X-X_0}{t_0},N\rangle +H=\lambda$.
\end{proof}

\begin{definition}
Let $X: M\rightarrow \mathbb{R}^{n+1}$ be an $n$-dimensional immersed hypersurface in the Euclidean space
$\mathbb{R}^{n+1}$. If $\langle \frac{X-X_0}{t_0},N\rangle +H=\lambda$ holds, we call
 $X: M\rightarrow \mathbb{R}^{n+1}$ {\it a $\lambda$-hypersurface of the weighted volume-preserving mean curvature flow}.
\end{definition}
\begin{remark} If  $\lambda=0$,  then the $\lambda$-hypersurface is a self-shrinker of the mean curvature
flow. Hence, we know that the notation of the $\lambda$-hypersurface is a generalization of the self-shrinker.
\end{remark}
\begin{theorem}
Let $X:M\rightarrow \mathbb{R}^{n+1}$ be an immersed hypersurface.
The following statements are equivalent with each other:

\begin{enumerate}
\item $X:M\rightarrow \mathbb{R}^{n+1}$ is a $\lambda$-hypersurface.
\item  $X:M\rightarrow \mathbb{R}^{n+1}$ is a critical point of the weighted area functional $A(t)$ for  all weighted volume-preserving variations.
\item  $X:M\rightarrow \mathbb{R}^{n+1}$ is a hypersurface with constant weighted mean curvature $H_w=\lambda$  in $\mathbb {R}^{n+1}$
with respect to the metric $g_{AB}=e^{-\frac{|X-X_0|^2}{nt_0}}\delta_{AB}$, where the weighted mean curvature and the mean curvature $H$ are related
by  $H_w =e^{-\frac{|X-X_0|^2}{2nt_0}}H$.
\end{enumerate}

\end{theorem}

\begin{example}
The $n$-dimensional sphere $S^n(r)$ with radius $r>0$ is a compact $\lambda$-hypersurface in $\mathbb{R}^{n+1}$
with $\lambda=\frac{n}r-r$. It  should be  remarked  that the sphere $S^n(\sqrt n)$ is the only self-shrinker sphere in $\mathbb{R}^{n+1}$.
\end{example}
\begin{example}
For $1\leq k\leq n-1$, the $n$-dimensional cylinder $S^k(r)\times \mathbb{R}^{n-k}$ with radius $r>0$
 is a complete and non-compact $\lambda$-hypersurface
in $\mathbb{R}^{n+1}$ with $\lambda=\frac{k}r-r$. We should notice that the cylinder $S^k(\sqrt k)\times \mathbb{R}^{n-k}$ is the only self-shrinker
cylinder in $\mathbb{R}^{n+1}$.
\end{example}

\noindent  From \cite{[C]}, Chang has proved there exist a lot of complete embedded $\lambda$-curves $\Gamma$ in $\mathbb{R}^{2}$. Hence we have

\begin{example}
The $n$-dimensional hypersurfaces $\Gamma \times \mathbb{R}^{n-1}$ are  complete embedded $\lambda$-hypersurfaces, which are not self-shrinkers, in $\mathbb{R}^{n+1}$. 
\end{example}

\begin{remark}
For $1$-dimensional self-shrinker in $\mathbb{R}^{2}$, Abresch and Langer \cite{[AL]} proved the circle is 
the only compact embedded self-shrinker. But for $\lambda$-curve in $\mathbb{R}^{2}$, Chang \cite{[C]} has proved, for  $\lambda<0$, there are many compact embedded $\lambda$-curves other than the circle. From the above examples, we know that there are 
a lot of  examples of complete embedded $\lambda$-hypersurfaces, which are not self-shrinkers,  in $\mathbb{R}^{n+1}$.
\end{remark}

\begin{proposition} \label{prop 2.2}Let $X: M\rightarrow \mathbb{R}^{n+1}$ be a $\lambda$-hypersurface in the Euclidean space
$\mathbb{R}^{n+1}$. If the mean curvature $H$ is constant, then $X: M\rightarrow \mathbb{R}^{n+1}$
is isometric to   $S^k(r)\times \mathbb{R}^{n-k}$, $0\leq k\leq n$, locally.
\end{proposition}
\begin{proof}
Since  $X: M\rightarrow \mathbb{R}^{n+1}$ is a $\lambda$-hypersurface, we have
$\langle X, N\rangle +H=\lambda$. If $H$ is constant, we get, for any $1\leq i\leq n$,
$$
\nabla_i\langle X, N\rangle=-\lambda_i\langle X, e_i\rangle=0,
$$
where $\lambda_i$ is the principal curvature of the $\lambda$-hypersurface. If $\lambda_{i_0}\neq 0$ at a point $p$ for some
$i_0$, there exists a neighborhood $U$ of $p$ such that $\lambda_{i_0}\neq 0$ in $U$.
Hence, we know $\langle X, e_{i_0}\rangle=0$ in $U$.
Thus,
$$
X=\sum_{j\neq i_0}\langle X, e_{j}\rangle e_j+\langle X,N\rangle N.
$$
We obtain
$$
e_{i_0}=\nabla_{i_0}X=-\langle X,N\rangle \lambda_{i_0} e_{i_0},
$$
that is,
$\lambda_{i_0}(H-\lambda)=1$ is constant. Thus, on $U$, $\lambda_{i_0}$ is constant.
Therefore,  the $\lambda$-hypersurface is isoparametric.  We obtain that
$X: M\rightarrow \mathbb{R}^{n+1}$
is isometric to   $S^k(r)\times \mathbb{R}^{n-k}$, $0\leq k\leq n$, locally.
\end{proof}
\begin{definition}
A family of $n$-dimensional immersed hypersurfaces $X(t): M\rightarrow \mathbb{R}^{n+1}$ in the Euclidean space
$\mathbb{R}^{n+1}$
is called a self-similar solution
of the weighted volume-preserving mean curvature flow if $X(t)=\beta(t)X$ holds, where $\beta(t)>0$.
\end{definition}
\begin{proposition}
A family of $n$-dimensional immersed hypersurfaces $X(t): M\rightarrow \mathbb{R}^{n+1}$ in the Euclidean space
$\mathbb{R}^{n+1}$
is  a self-similar solution
of the weighted volume-preserving mean curvature flow if and only if   $X(t)=\sqrt{1+\beta_0t}X$, where $\beta_0$ is
a constant.
\end{proposition}
\begin{proof}
If $X(t): M\rightarrow \mathbb{R}^{n+1}$ is  a self-similar solution
of the weighted volume-preserving mean curvature flow, we have  $X(t)=\beta(t)X$.
Hence, the mean curvature $H(t)$ of $X(t)$ satisfies
$$
H(t)=\dfrac{H}{\beta(t)}.
$$
Thus,
$$
\alpha(t) =\dfrac{\int_MH(t)\langle N(t), N\rangle e^{-\frac{|X|^2}2}d\mu}{\int_M\langle N(t), N\rangle e^{-\frac{|X|^2}2}d\mu}
=\dfrac{\int_MH e^{-\frac{|X|^2}2}d\mu}{\beta(t)\int_M e^{-\frac{|X|^2}2}d\mu}.
$$
From the equation of the weighted volume-preserving mean curvature flow, we have
\begin{equation}\label{eq:2008}
\dfrac{\partial \beta(t)}{\partial t}X^{\perp}=\dfrac1{\beta(t)}\bigl(-\alpha(0) N +\mathbf{H}\bigl).
\end{equation}
We obtain
$
\dfrac{\partial \beta(t)^2}{\partial t}=\beta_0={\rm constant}.
$
Since $\beta(0)=1$, we have $\beta(t)=\sqrt{1+\beta_0t}$.

\noindent
The inverse is obvious.
\end{proof}

\begin{proposition} Let $X: M\rightarrow \mathbb{R}^{n+1}$ be a $\lambda$-hypersurface in the Euclidean space
$\mathbb{R}^{n+1}$.
If $X(t)=\sqrt{1+\beta_0t}X$ is  a self-similar solution of the weighted volume-preserving mean curvature flow,
then   $X: M\rightarrow \mathbb{R}^{n+1}$
is isometric to   $S^k(r)\times \mathbb{R}^{n-k}$, $0\leq k\leq n$, locally or
$V(0)=0$ and $\beta_0=-2$.
\end{proposition}
\begin{proof}
Since  $X: M\rightarrow \mathbb{R}^{n+1}$ is a $\lambda$-hypersurface, we have
$\langle X, N\rangle +H=\lambda$
and
$$
V(t)=\int_M \langle X(t), N\rangle  e^{-\frac{|X|^2}2}d\mu=\sqrt{1+\beta_0t}V(0).
$$
Since $X(t)=\sqrt{1+\beta_0t}X$ is  a self-similar solution of the weighted volume-preserving mean curvature flow,
then $\beta_0=0$  or $V(0)=0$. If $\beta_0=0$, then $H$ is constant from (\ref{eq:2008}). According to the proposition \ref{prop 2.2},
we know that  $X: M\rightarrow \mathbb{R}^{n+1}$
is isometric to  $S^k(r)\times \mathbb{R}^{n-k}$, $0\leq k\leq n$, locally.
If $\beta_0\neq 0$, we have $V(0)=0$ since $V(t)$ is constant. The (\ref{eq:2008}) gives $\beta_0=-2$.

\end{proof}

\begin{definition}
If $X: M\rightarrow \mathbb{R}^{n+1}$ is an $n$-dimensional hypersurface in $\mathbb{R}^{n+1}$,
we say that $M$ has polynomial area growth if there exist constant $C$ and $d$ such that for all $r\geq 1$,
\begin{equation}
{\rm Area}(B_r(0)\cap X(M))=\int_{B_r(0)\cap X(M)}d\mu\leq Cr^d,
\end{equation}
where $B_r(0)$ is a standard ball in $\mathbb{R}^{n+1}$ with radius $r$ and centered at the origin.
\end{definition}
\vskip 5mm

\section{Properties of $\lambda$-hypersurfaces}
\noindent
In this section, we give several properties of $\lambda$-hypersurfaces.
We define an elliptic operator $\mathcal{L}$ by
\begin{equation}\label{eq:12-6-1}
\mathcal{L}f=\Delta f-\langle X,\nabla f\rangle,
\end{equation}
where $\Delta$ and $\nabla$ denote the Laplacian and the gradient operator of the $\lambda$-hypersurface, respectively.
We should notice that the $\mathcal{L}$ operator was introduced by Colding and Minicozzi in \cite{[CM]}
for self-shrinkers.

\noindent
By a direct calculation, for a constant
vector $a\in \mathbb{R}^{n+1}$, we have

\begin{equation*}
\aligned
 \mathcal{L}\langle X,a\rangle &=\Delta \langle X,a\rangle -\langle X,\nabla\langle X,a\rangle \rangle \\
 &=\sum_i\langle X,a\rangle _{,ii}-\sum_i\langle X,a\rangle _{,i}\langle X,e_i\rangle \\
&=\langle HN,a\rangle -\sum_i\langle e_i,a\rangle \langle X,e_i\rangle \\
&=\langle HN,a\rangle -\langle X,a\rangle +\langle X,N\rangle \langle N,a\rangle \\
&=\lambda\langle N,a\rangle -\langle X,a\rangle ,
\endaligned
\end{equation*}

\begin{equation*}
\aligned
 \mathcal{L}\langle N,a\rangle &=\sum_i\langle N,a\rangle _{,ii}-\sum_i\langle N,a\rangle _{,i}\langle X,e_i\rangle \\
&=\langle -H_{,i}e_i-SN,a\rangle +\sum_i\langle X,e_i\rangle \langle \sum_j h_{ij}e_j,a\rangle \\
&=\langle X,N\rangle _{,i}\langle e_i,a\rangle -\langle SN,a\rangle +\sum_i\langle X,e_i\rangle \langle \sum_j h_{ij}e_j,a\rangle \\
&=-S\langle N,a\rangle ,
\endaligned
\end{equation*}
where $S=\sum_{i,j}h_{ij}^2$ is the squared norm of the second fundamental form.

\begin{equation*}
\aligned
 \frac{1}{2}\mathcal{L}(|X|^2)&=\langle \Delta X,X\rangle +\sum_i\langle X_{,i},X_{,i}\rangle -\sum_i\langle X,e_i\rangle \langle X,e_i\rangle \\
&=n-|X|^2+\lambda\langle X,N\rangle.
\endaligned
\end{equation*}
Hence,  we have the following

\begin{lemma}\label{lemma 0}
If  $X: M\rightarrow \mathbb{R}^{n+1}$ is  a $\lambda$-hypersurface, then we have
\begin{equation}\label{eq:2001}
\aligned
 \mathcal{L}\langle X,a\rangle &=\lambda\langle N,a\rangle -\langle X,a\rangle,
\endaligned
\end{equation}

\begin{equation}\label{eq:2002}
\aligned
 \mathcal{L}\langle N,a\rangle &=-S\langle N,a\rangle,
\endaligned
\end{equation}
\begin{equation}\label{eq:2003}
\aligned
 \frac{1}{2}\mathcal{L}(|X|^2)&=n-|X|^2+\lambda\langle X,N\rangle .
\endaligned
\end{equation}
\end{lemma}

\noindent
The following lemma\label{lemma 1} due to Colding and Minicozzi \cite{[CM]} is needed in order to prove our results.
\begin{lemma}\label{lemma 1}
If $X: M\rightarrow \mathbb{R}^{n+1}$ is a hypersurface, $u$ is a $C^1$-function with compact support and
$v$ is a $C^2$-function, then
\begin{equation}
\int_M u(\mathcal{L}v)e^{-\frac{|X|^2}{2}}d\mu=-\int_M\langle \nabla u,\nabla v\rangle e^{-\frac{|X|^2}{2}}d\mu.
\end{equation}
\end{lemma}

\begin{corollary}\label{corollary 1}
Let $X: M\rightarrow \mathbb{R}^{n+1}$ be a complete hypersurface. If $u$, $v$ are $C^2$ functions satisfying
\begin{equation}
\int_M(|u\nabla v|+|\nabla u||\nabla v|+|u\mathcal{L}v|)e^{-\frac{|X|^2}{2}}d\mu< +\infty,
\end{equation}
then we have
\begin{equation}
\int_M u(\mathcal{L}v)e^{-\frac{|X|^2}{2}}d\mu=-\int_M\langle \nabla u,\nabla v\rangle e^{-\frac{|X|^2}{2}}d\mu.
\end{equation}
\end{corollary}

\begin{lemma}\label{lemma 10}
Let $X: M\rightarrow \mathbb{R}^{n+1}$ be an $n$-dimensional complete $\lambda$-hypersurface with polynomial area growth, then
\begin{equation}\label{eq:1001}
 \int_M (\langle X,a\rangle -\lambda\langle N,a\rangle )e^{-\frac{|X|^2}{2}}d\mu=0,
\end{equation}


\begin{equation}\label{eq:1003}
 \int_M \bigl(n-|X|^2+\lambda\langle X,N\rangle \bigl)e^{-\frac{|X|^2}{2}}d\mu=0,
\end{equation}

 \begin{equation}\label{eq:1010}
\aligned
&\ \ \ \int_M \langle X, a\rangle |X|^2e^{-\frac{|X|^2}{2}}d\mu\\
&=\int_M \biggl(2n\lambda\langle N, a\rangle +2\lambda\langle X,a\rangle (\lambda-H) -\lambda\langle N,a\rangle |X|^2\biggl)e^{-\frac{|X|^2}{2}}d\mu,
\endaligned
\end{equation}

 \begin{equation}\label{eq:1011}
\int_M\langle X,a\rangle ^2e^{-\frac{|X|^2}{2}}d\mu=\int_M\biggl(|a^{T}|^2+\lambda\langle N,a\rangle \langle X,a\rangle \biggl)e^{-\frac{|X|^2}{2}}d\mu,
\end{equation}
where $a^{T}=\sum_i<a,e_i>e_i$.

\begin{equation}\label{eq:1012}
\aligned
&\int_M\biggl(|X|^2-n-\frac{\lambda(\lambda-H)}{2}\biggl)^2e^{-\frac{|X|^2}{2}}d\mu\\
&=\int_M\biggl\{(\frac{\lambda^2}{4}-1)(\lambda-H)^2 +2n-H^2+\lambda^2\biggl\}e^{-\frac{|X|^2}{2}}d\mu.
\endaligned
\end{equation}

\end{lemma}

\begin{proof}
Equations \eqref{eq:1001}
and \eqref{eq:1003} just follow
from the corollary \ref{corollary 1} and equations \eqref{eq:2001},
and \eqref{eq:2003}.
Since $X: M\rightarrow \mathbb{R}^{n+1}$ is an $n$-dimensional complete $\lambda$-hypersurface with polynomial area growth,
by making use of $u=|X|^2$, $v=\langle X,a\rangle $ in the lemma \ref{lemma 1}, we have

\begin{equation*}
\aligned
&\int_M \langle X, a\rangle |X|^2e^{-\frac{|X|^2}{2}}d\mu\\
&=-\int_M \mathcal{L}\langle X,a\rangle |X|^2e^{-\frac{|X|^2}{2}}d\mu
+\int_M\lambda \langle N,a\rangle |X|^2e^{-\frac{|X|^2}{2}}d\mu\\
                  &=-\int_M\langle X,a\rangle \mathcal{L}|X|^2e^{-\frac{|X|^2}{2}}d\mu+\int_M\lambda \langle N,a\rangle |X|^2e^{-\frac{|X|^2}{2}}d\mu\\
                  &=-\int_M 2\langle X,a\rangle \bigl[n+\lambda\langle X,N\rangle -|X|^2\bigl]e^{-\frac{|X|^2}{2}}d\mu
                  +\int_M\lambda \langle N,a\rangle |X|^2e^{-\frac{|X|^2}{2}}d\mu\\
                  &=2\int_M \langle X, a\rangle |X|^2e^{-\frac{|X|^2}{2}}d\mu-2n\int_M\langle X,a\rangle -2\lambda\langle X,a\rangle (\lambda-H)e^{-\frac{|X|^2}{2}}d\mu\\
                  &\ \ \ +\int_M\lambda \langle N,a\rangle |X|^2e^{-\frac{|X|^2}{2}}d\mu.
\endaligned
\end{equation*}
Hence, it follows that
\begin{equation*}
\aligned
&\ \ \ \int_M \langle X, a\rangle |X|^2e^{-\frac{|X|^2}{2}}d\mu\\
&=\int_M \biggl(2n\lambda\langle N, a\rangle +2\lambda\langle X,a\rangle (\lambda-H) -\lambda\langle N,a\rangle |X|^2\biggl)e^{-\frac{|X|^2}{2}}d\mu.
\endaligned
\end{equation*}
Taking  $u=v=\langle X,a\rangle $ in the lemma \ref{lemma 1}, we can get \eqref{eq:1011}.
Putting $u=v=|X|^2$ in the lemma \ref{lemma 1}, we can have
\begin{equation*}
\aligned
& \int_M \lambda(\lambda-H)|X|^2e^{-\frac{|X|^2}{2}}d\mu\\
&=\int_M(|X|^4-n|X|^2
 +\frac{1}{2}|X|^2\mathcal{L}|X|^2)e^{-\frac{|X|^2}{2}}d\mu\\
 &=\int_M (|X|^4-n|X|^2)e^{-\frac{|X|^2}{2}}d\mu-\int_M\frac{1}{2}\langle \nabla |x|^2,\nabla |x|^2\rangle e^{-\frac{|X|^2}{2}}d\mu\\
 &=\int_M \bigl(|X|^4-(n+2)|X|^2+2(\lambda-H)^2\bigl)e^{-\frac{|X|^2}{2}}d\mu,
\endaligned
\end{equation*}
that is,
\begin{equation*}
\int_M \biggl\{|X|^4-[n+\lambda(\lambda-H)]|X|^2-2|X|^2+2(\lambda-H)^2\biggl\}e^{-\frac{|X|^2}{2}}d\mu=0.
\end{equation*}
Thus, we have

\begin{equation*}
\aligned
&0=\int_M \biggl\{|X|^4-2[n+\frac{(\lambda-H)\lambda}{2}]|X|^2+n^2
 +n\lambda(\lambda-H)\\
 &\ \ \ \ \ \ \ \ \ \ -2n-2\lambda(\lambda-H)+2(\lambda-H)^2\biggl\}e^{-\frac{|X|^2}{2}}d\mu\\
 &=\int_M\biggl\{\bigl[|X|^2-(n+\frac{\lambda(\lambda-H)}{2})\bigl]^2
    -\frac{\lambda^2(\lambda-H)^2}{4}+2(\lambda-H)^2\\
 &\ \ \ \ \ \ \ \ \ \ -2n-2\lambda(\lambda-H)\biggl\}e^{-\frac{|X|^2}{2}}d\mu\\
 &=\int_M\biggl\{\bigl(|X|^2-n-\frac{\lambda(\lambda-H)}{2}\bigl)^2-(\frac{\lambda^2}{4}-1)(\lambda-H)^2
    -2n+H^2-\lambda^2\biggl\}e^{-\frac{|X|^2}{2}}d\mu,
\endaligned
\end{equation*}
namely,
\begin{equation*}
\aligned
&\int_M\biggl(|X|^2-n-\frac{\lambda(\lambda-H)}{2}\biggl)^2e^{-\frac{|X|^2}{2}}d\mu\\
&=\int_M\biggl\{(\frac{\lambda^2}{4}-1)(\lambda-H)^2 +2n-H^2+\lambda^2\biggl\}e^{-\frac{|X|^2}{2}}d\mu.
\endaligned
\end{equation*}
\end{proof}

\vskip 5mm

\section{A classification of compact $\lambda$-hypersurfaces}

\noindent
In this section, we will give a classification of compact $\lambda$-hypersurfaces.
First of all, we give some lemmas.
\begin{lemma}\label{lemma 2}
Let $X: M\rightarrow \mathbb{R}^{n+1}$ be an $n$-dimensional $\lambda$-hypersurface. Then,
the following holds.
\begin{equation}
\mathcal{L}H=H+S(\lambda-H),
\end{equation}

\begin{equation}\label{eq:224-6}
\frac{1}{2}\mathcal{L}S=\sum_{i,j,k}h_{ijk}^2+(1-S)S+\lambda f_3,
\end{equation}

\begin{equation}\label{eq:224-9}
\mathcal{L}\sqrt{S}=\frac{1}{\sqrt{S}}\biggl(\sum_{i,j,k}h_{ijk}^2-|\nabla\sqrt{S}|^2\biggl)+\sqrt{S}(1-S)
 +\frac{1}{\sqrt{S}}\lambda f_3,
\end{equation}

\begin{equation}\label{eq:224-10}
\mathcal{L}\log(H-\lambda)=1-S+\frac{\lambda}{H-\lambda}-|\nabla \log(H-\lambda)|^2,\ \ \ \ {\rm if}\ H-\lambda>0,
\end{equation}
where $f_3=\sum_{i,j,k}h_{ij}h_{jk}h_{ki}$.
\end{lemma}
\vskip 3pt\noindent {\it Proof}.
Since $\langle X,N\rangle +H=\lambda$, one has
\begin{equation}\label{eq:224-3}
H_{,i}=\sum_j h_{ij}\langle X,e_j\rangle,
\end{equation}
\begin{equation*}
H_{,ik}=\sum_jh_{ijk}\langle X,e_j\rangle +h_{ik}+\sum_jh_{ij}h_{jk}(\lambda-H).
\end{equation*}
Hence,
\begin{equation}\label{eq:224-11}
\Delta H=\sum_iH_{,ii}=\sum_iH_{,i}\langle X,e_i\rangle +H+S(\lambda-H)
\end{equation}
and
\begin{equation*}
\aligned
\mathcal{L}H&=\Delta H-\sum_i\langle X,e_i\rangle H_{,i}=H+S(\lambda-H).
\endaligned
\end{equation*}
By a direct calculation, we have from (\ref{eq:12-6-6})
\begin{equation*}
\aligned
\mathcal{L}h_{ij}&=\Delta h_{ij}-\sum_k\langle X,e_k\rangle h_{ijk}\\
&=(1-S)h_{ij}+\lambda\sum_kh_{ik}h_{kj}.
\endaligned
\end{equation*}
Then it follows that
\begin{equation*}
\aligned
\frac{1}{2}\mathcal{L}S&=\frac{1}{2}\biggl(\Delta \sum_{i,j}h_{ij}^2-\sum_k\langle X,e_k\rangle \bigl(\sum_{i,j}h_{ij}^2\bigl)_{,k}\biggl)\\
&=\sum_{i,j,k}h_{ijk}^2+(1-S)S+\lambda f_3.
\endaligned
\end{equation*}
Since
\begin{equation}
\mathcal{L}S=2|\nabla\sqrt{S}|^2+2\sqrt{S}\mathcal{L}\sqrt{S},
\end{equation}
we have

\begin{equation*}
\aligned
\mathcal{L}\sqrt{S}&=\frac{1}{2\sqrt{S}}\mathcal{L}S-\frac{|\nabla \sqrt{S}|^2}{\sqrt{S}}\\
                   &=\frac{1}{\sqrt{S}}\biggl(\sum_{i,j,k}h_{ijk}^2-|\nabla\sqrt{S}|^2\biggl)+\sqrt{S}(1-S)
 +\frac{1}{\sqrt{S}}\lambda f_3.
\endaligned
\end{equation*}
\begin{equation*}
\aligned
\mathcal{L}\log(H-\lambda)&=\Delta\log(H-\lambda)-\sum_i\langle X,e_i\rangle (\log(H-\lambda))_{,i}\\
&=\frac{1}{H-\lambda}\mathcal{L}H-|\nabla \log(H-\lambda)|^2\\
&=1-S+\frac{\lambda}{H-\lambda}-|\nabla \log(H-\lambda)|^2.
\endaligned
\end{equation*}
We complete the proof of the lemma.
$$\eqno{\Box}$$

\begin{theorem}\label{theorem 4.1}
Let $X: M\rightarrow \mathbb{R}^{n+1}$  be an $n$-dimensional compact $\lambda$-hypersurface in $\mathbb{R}^{n+1}$.
If $H-\lambda\geq0$ and  $\lambda(f_3(H-\lambda)-S)\geq 0$, then $X: M\rightarrow \mathbb{R}^{n+1}$
is isometric to a round sphere $S^{n}(r)$ with $\lambda =\frac n r- r$.
\end{theorem}

\vskip 3pt\noindent {\it Proof}.
Since
\begin{equation*}
\mathcal{L}H=H+S(\lambda-H)
\end{equation*}
and
$$H-\lambda\geq0,$$
we have
\begin{equation*}
\mathcal{L}H-H\leq 0.
\end{equation*}
If $\lambda\leq 0$, we conclude from the maximum principle that either $H\equiv\lambda$ or $H-\lambda> 0$.
 If $H\equiv\lambda$,  \eqref{eq:224-11} gives that $H=\lambda=0$ and $M$ is a self-shrinker, it is impossible since M is compact; If $\lambda>0$, we have $f_3(H-\lambda)-S\geq0$. In this case, if $H-\lambda=0$ at some point $p\in M$, then $S=0$ and $H=\lambda=0$ at $p$, that is $\lambda\equiv0$ and $M$ is self-shrinker, it is also impossible since M is compact. Hence for any $\lambda$, we have $H-\lambda> 0$.

\noindent From the lemma \ref{lemma 2}, we can get

\begin{equation*}
\aligned
\mathcal{L}\frac{1}{(H-\lambda)^2}&=\Delta\frac{1}{(H-\lambda)^2}-\sum_i\langle X,e_i\rangle \bigl(\frac{1}{(H-\lambda)^2}\bigl)_{,i}\\
&=\frac{6}{(H-\lambda)^4}|\nabla (H-\lambda)|^2-\frac{2}{(H-\lambda)^3}[H-S(H-\lambda)]
\endaligned
\end{equation*}
and
\begin{equation*}
\aligned
\mathcal{L}\frac{S}{(H-\lambda)^2}&=\Delta\frac{S}{(H-\lambda)^2}-\sum_i\langle X,e_i\rangle \bigl(\frac{S}{(H-\lambda)^2}\bigl)_{,i}\\
&=\frac{1}{(H-\lambda)^2}\mathcal{L}S+2\langle \nabla S,\nabla
  (\frac{1}{(H-\lambda)^2})\rangle +S\mathcal{L}(\frac{1}{(H-\lambda)^2})\\
&=\frac{2}{(H-\lambda)^2}\biggl(\sum_{i,j,k}h_{ijk}^2+(1-S)S+\lambda f_3\biggl)
  +2\langle \nabla S,\nabla(\frac{1}{(H-\lambda)^2})\rangle \\
&\ \ +S\biggl(\frac{6}{(H-\lambda)^4}|\nabla (H-\lambda)|^2-\frac{2}{(H-\lambda)^3}[H-S(H-\lambda)]\biggl).
\endaligned
\end{equation*}
By multiplying $Se^{-\frac{|X|^2}{2}}$ in the above equation and using
\begin{equation*}
\int_M S\mathcal{L}\frac{S}{(H-\lambda)^2}e^{-\frac{|X|^2}{2}}d\mu
=-\int_M \langle \nabla S,\nabla(\frac{S}{(H-\lambda)^2})\rangle e^{-\frac{|X|^2}{2}}d\mu,
\end{equation*}
one has

\begin{equation}\label{eq:4-11-1}
\aligned
&\ \ 2\int_M\frac{S}{(H-\lambda)^4}\sum_{i,j,k}|h_{ijk}(H-\lambda)-h_{ij}H_{,k}|^2e^{-\frac{|X|^2}{2}}d\mu\\
&
+\int_M |\nabla (\frac{S}{(H-\lambda)^2})|^2(H-\lambda)^2e^{-\frac{|X|^2}{2}}d\mu\\
&+2\int_M\frac{S}{(H-\lambda)^2}\lambda\biggl(f_3-\frac{S}{H-\lambda}\biggl)e^{-\frac{|X|^2}{2}}d\mu=0.
\endaligned
\end{equation}
Then it follows from $\lambda(f_3(H-\lambda)-S)\geq 0$ that
\begin{equation}
\lambda(f_3-\frac{S}{H-\lambda})=0, \ \ \frac{S}{(H-\lambda)^2}={\rm constant},\ \ h_{ijk}(H-\lambda)=h_{ij}H_{,k}.
\end{equation}
We next consider two cases.

{\bf Case 1:  $\lambda=0$}

\noindent
In this case, we know $M$ is isometric to $S^n(\sqrt{n})$ from Huisken's result \cite{[H2]}.

{\bf Case 2:  $\lambda\neq 0$}

\noindent
In this case, one gets
$$
f_3-\frac{S}{H-\lambda}=0,\ \ h_{ijk}(H-\lambda)=h_{ij}H_{,k}.
$$
If $H$ is constant, then $h_{ijk}=0$, thus $M$ is $S^n(r)$ by the result of Lawson \cite{[L]}.

\noindent
If $H$ is not constant, then there exists a neighborhood $U$ such that $|\nabla H|\neq 0$ on $U$.
We can choose $e_1, \cdots, e_n$ such that $e_1=\frac{\nabla H}{|\nabla H|}$. It follows from $h_{ijk}=h_{ikj}$ that $h_{ij}H_{,k}=h_{ik}H_{,j}$ and
\begin{equation*}
\aligned
&0=\sum_{i,j,k}|h_{ij}H_{,k}-h_{ik}H_{,j}|^2\\
&\ \ =2S|\nabla H|^2-2\sum_ih_{1i}^2|\nabla H|^2\\
&\ \ =2|\nabla H|^2(S-\sum_i h_{1i}^2),
\endaligned
\end{equation*}
that is,
$$
\sum_{i=1}^nh_{1i}^2=S=h_{11}^2+2\sum_{j\neq1}^nh_{1j}^2+\sum_{k,l\geq2}h_{kl}^2.
$$
Therefore,  $S=h_{11}^2=H^2$ on $U$. On the other hand, we see from $\frac{S}{(H-\lambda)^2}={\rm constant}$ that $H$ is constant on $U$.
It is a contradiction. The proof of the theorem \ref{theorem 4.1} is completed.
$$\eqno{\Box}$$

\begin{remark}  The assumption  $\lambda(f_3(H-\lambda)-S)\geq 0$ in the theorem \ref{theorem 4.1} is satisfied
for self-shrinkers of the mean curvature flow, automatically. When $\lambda>0$,  this condition is needed in order to prove
$H>\lambda$ since the maximum  principle does not work for this case. We  think that  the assumption is essential.
In particular, for case of complete and non-compact $\lambda$-hypersurfaces, this condition is essential in
section 8. In fact, $\Gamma \times \mathbb{R}^{n-1}$ are counterexamples since $H-\lambda>0$, where $\Gamma$ are compact embedded $\lambda$-curves other than the circle (see Remark 2.2). It is a very interesting problem to construct  counterexamples for compact case.
\end{remark}

\vskip 5mm

\section{The first variation of $\mathcal{F}$-functional}

\noindent
In this section, we will give another  variational characterization of $\lambda$-hypersurfaces. Let $X(s): M\rightarrow \mathbb{R}^{n+1}$ be immersions with $X(0)=X$. The variation vector field $\frac{\partial}{\partial s}X(s)|_{s=0}$
is the normal variation vector field $fN$.

\noindent
For $X_0\in \mathbb{R}^{n+1}$ and a real number $t_0 $, the $\mathcal{F}$-functional is defined by
\begin{equation*}
\aligned
&\ \ \ \ \mathcal{F}_{X_s,t_s}(s)=\mathcal{F}_{X_s,t_s}(X(s))\\
&=(4\pi t_s)^{-\frac{n}{2}}\int_M e^{-\frac{|X(s)-X_s|^2}{2t_s}}d\mu_s
  +\lambda (4\pi t_0)^{-\frac{n}{2}}(\frac{t_0}{t_s})^{\frac{1}{2}}\int_M \langle X(s)-X_s,N\rangle e^{-\frac{|X-X_0|^2}{2t_0}}d\mu,
\endaligned
\end{equation*}
where $X_s$ and $t_s$ denote the variations of $X_0$ and $t_0$.
Let
\begin{equation*}
\frac{\partial t_s}{\partial s}=h(s),\ \ \frac{\partial X_s}{\partial s}=y(s),\ \
\frac{\partial X(s)}{\partial s}=f(s)N(s),
\end{equation*}
one calls that  $X: M\rightarrow \mathbb{R}^{n+1}$ is {\it a critical point of } $\mathcal{F}_{X_s,t_s}(s)$
if it is critical with respect to all normal variations and all variations in $X_0$ and $t_0$.

\begin{lemma}\label{lemma 3}
Let $X(s)$ be a variation of $X$ with normal variation vector field $\frac{\partial X(s)}{\partial s}|_{s=0}=fN$. If $X_s$ and $t_s$ are variations of $X_0$ and $t_0$ with $\frac{\partial X_s}{\partial s}|_{s=0}=y$ and $\frac{\partial t_s}{\partial s}|_{s=0}=h$, then the first variation formula of $\mathcal{F}_{X_s,t_s}(s)$ is given by
\begin{equation}\label{eq:1000}
\aligned
&\mathcal{F}^{'}_{X_0,t_0}(0)\\ &=(4\pi t_0)^{-\frac{n}{2}}\int_M \biggl(\lambda-(H+\langle \frac{X-X_0}{t_0},N\rangle )\biggl)f e^{-\frac{|X-X_0|^2}{2}}d\mu\\
&\ \ +(4\pi t_0)^{-\frac{n}{2}}\int_M \biggl(\langle \frac{X-X_0}{t_0},y\rangle -\lambda \langle N,y\rangle \biggl)e^{-\frac{|X-X_0|^2}{2}}d\mu\\
&\ \ +(4\pi t_0)^{-\frac{n}{2}}\int_M \biggl(\frac{|X-X_0|^2}{t_0}-n-\lambda\langle X-X_0,N\rangle \biggl)\frac{h}{2t_0}e^{-\frac{|X-X_0|^2}{2}}d\mu.
\endaligned
\end{equation}
\end{lemma}
\vskip 3pt\noindent {\it Proof}. Defining
\begin{equation}
\mathbb{A}(s)=\int_M e^{-\frac{|X(s)-X_s|^2}{2t_s}}d\mu_s, \ \
\mathbb{V}(s)=\int_M \langle X(s)-X_s,N\rangle e^{-\frac{|X-X_0|^2}{2t_0}}d\mu,
\end{equation}
then
\begin{equation*}
\aligned
\mathcal{F}^{'}_{X_s,t_s}(s)&=(4\pi t_s)^{-\frac{n}{2}}\mathbb{A}^{'}(s)+\lambda (4\pi t_0)^{-\frac{n}{2}}(\frac{t_0}{t_s})^{\frac{1}{2}}\mathbb{V}^{'}(s)\\
&\ \ -(4\pi t_s)^{-\frac{n}{2}}\frac{n}{2t_s}h \mathbb{A}(s)-\lambda (4\pi t_0)^{-\frac{n}{2}}(\frac{t_0}{t_s})^{\frac{1}{2}}
   \frac{h}{2t_s}\mathbb{V}(s).
\endaligned
\end{equation*}
Since
\begin{equation*}
\aligned
\mathbb{A}^{'}(s)&=\int_M \biggl\{-\langle \frac{X(s)-X_s}{t_s}, \frac{\partial X(s)}{\partial s}-\frac{\partial X_s}{\partial s}\rangle  +\frac{|X(s)-X_s|^2}{2t_s^2}h\\
&\ \ \ \ \ \ \ \ \ \ -H_s\langle \frac{\partial X(s)}{\partial s}, N(s)\rangle \biggl\}
    e^{-\frac{|X(s)-X_s|^2}{2t_s}}d\mu_s,
\endaligned
\end{equation*}

\begin{equation*}
\mathbb{V}^{'}(s)=\int_M\langle \frac{\partial X(s)}{\partial s}-\frac{\partial X_s}{\partial s}, N\rangle e^{-\frac{|X-X_0|^2}{2t_0}}d\mu,
\end{equation*}
we have
\begin{equation*}
\aligned
&\ \ \ \ \mathcal{F}^{'}_{X_s,t_s}(s)\\
&=(4\pi t_s)^{-\frac{n}{2}}\int_M -(H_s+\langle \frac{X(s)-X_s}{t_s}, N(s)\rangle )fe^{-\frac{|X(s)-X_s|^2}{2t_s}}d\mu_s\\
&\ \ \ +(4\pi t_0)^{-\frac{n}{2}}\sqrt{\frac{t_0}{t_s}}\int_M \lambda f\langle N(s),N\rangle e^{-\frac{|X-X_0|^2}{2t_0}}d\mu\\
&\ \ \ +(4\pi t_s)^{-\frac{n}{2}}\int_M \langle \frac{X(s)-X_s}{t_s},y\rangle e^{-\frac{|X(s)-X_s|^2}{2t_s}}d\mu_s\\
&\ \ \ +(4\pi t_0)^{-\frac{n}{2}}\sqrt{\frac{t_0}{t_s}}\int_M \lambda \langle -y,N\rangle e^{-\frac{|X-X_0|^2}{2t_0}}d\mu\\
&\ \ \ +(4\pi t_s)^{-\frac{n}{2}}\int_M (-\frac{n}{2t_s}+\frac{|X(s)-X_s|^2}{2t_s^2})he^{-\frac{|X(s)-X_s|^2}{2t_s}}d\mu_s\\
&\ \ \ +(4\pi t_0)^{-\frac{n}{2}}\sqrt{\frac{t_0}{t_s}}\int_M -\frac{h\lambda}{2t_s}\langle X(s)-X_s,N\rangle e^{-\frac{|X-X_0|^2}{2t_0}}d\mu.
\endaligned
\end{equation*}
If $s=0$, then $X(0)=X$, $X_s=X_0$, $t_s=t_0$ and
\begin{equation*}
\aligned
&\ \ \ \ \mathcal{F}^{'}_{X_0,t_0}(0)\\
&=(4\pi t_0)^{-\frac{n}{2}}\int_M \biggl(\lambda-(H+\langle \frac{X-X_0}{t_0},N\rangle )\biggl)f e^{-\frac{|X-X_0|^2}{2}}d\mu\\
&\ \ \ +(4\pi t_0)^{-\frac{n}{2}}\int_M \biggl(\langle \frac{X-X_0}{t_0},y\rangle -\lambda \langle N,y\rangle \biggl)e^{-\frac{|X-X_0|^2}{2}}d\mu\\
&\ \ \ +(4\pi t_0)^{-\frac{n}{2}}\int_M \biggl(\frac{|X-X_0|^2}{t_0}-n-\lambda\langle X-X_0,N\rangle \biggl)\frac{h}{2t_0}e^{-\frac{|X-X_0|^2}{2}}d\mu.
\endaligned
\end{equation*}
$$\eqno{\Box}$$

\noindent
From the lemma \ref{lemma 3}, we know that if $X: M\rightarrow \mathbb{R}^{n+1}$ is a critical point of $\mathcal{F}$-functional
$\mathcal{F}_{X_s,t_s}(s)$, then
\begin{equation*}
H+\langle \frac{X-X_0}{t_0},N\rangle =\lambda.
\end{equation*}
We next prove that if $H+\langle \frac{X-X_0}{t_0},N\rangle =\lambda$, then $X: M\rightarrow \mathbb{R}^{n+1}$ must be a critical point of $\mathcal{F}$-functional $\mathcal{F}_{X_s,t_s}(s)$. For simplicity, we only consider the case of $X_0=0$ and $t_0=1$. In this case, $H+\langle \frac{X-X_0}{t_0},N\rangle =\lambda$ becomes
\begin{equation}
H+\langle X,N\rangle =\lambda.
\end{equation}

\noindent
Furthermore, we  know that $(M, X_0,t_0)$ is the critical point of the $\mathcal{F}$-functional if and only if $M$ is the critical point of $\mathcal{F}$-functional with
respect to fixed $X_0$ and $t_0$.

\begin{theorem}\label{theorem 4}
$X: M\rightarrow \mathbb{R}^{n+1}$ is a critical point of $\mathcal{F}_{X_s,t_s}(s)$ if and only if
$$
H+\langle \frac{X-X_0}{t_0},N\rangle =\lambda.
$$
\end{theorem}
\begin{proof}
We only prove the result for $X_0=0$ and $t_0=1$. In this case, the first variation formula \eqref{eq:1000} becomes
\begin{equation}\label{eq:1004}
\aligned
\mathcal{F}^{'}_{0,1}(0)&=(4\pi)^{-\frac{n}{2}}\int_M \biggl(\lambda-(H+\langle X, N\rangle )\biggl)f e^{-\frac{|X|^2}{2}}d\mu\\
&\ \ +(4\pi)^{-\frac{n}{2}}\int_M \biggl(\langle X,y\rangle -\lambda \langle N,y\rangle \biggl)e^{-\frac{|X|^2}{2}}d\mu\\
&\ \ +(4\pi)^{-\frac{n}{2}}\int_M \biggl(|X|^2-n-\lambda\langle X,N\rangle \biggl)\frac{h}{2}e^{-\frac{|X|^2}{2}}d\mu.
\endaligned
\end{equation}
If $X: M\rightarrow \mathbb{R}^{n+1}$ is a critical point of $\mathcal{F}_{0,1}$, then $X: M\rightarrow \mathbb{R}^{n+1}$
should satisfy $H+\langle X, N\rangle =\lambda$. Conversely, if  $H+\langle X, N\rangle =\lambda$ is satisfied,
then we know that $X: M\rightarrow \mathbb{R}^{n+1}$ is a $\lambda$-hypersurface. Therefore,
the last two terms in \eqref{eq:1004} vanish for any $h$ and any $y$ from \eqref{eq:1001}
and \eqref{eq:1003} of the lemma \ref{lemma 10}. Therefore $X: M\rightarrow \mathbb{R}^{n+1}$ is a critical point of $\mathcal{F}_{0,1}$.
\end{proof}

\begin{corollary}\label{corollary 5.1}
$X: M\rightarrow \mathbb{R}^{n+1}$ is a critical point of $\mathcal{F}_{X_s,t_s}(s)$ if and only if
$M$ is the critical point of $\mathcal{F}$-functional with respect to fixed $X_0$ and $t_0$.
\end{corollary}

\vskip 5mm

\section{The second variation of $\mathcal{F}$-functional}

\noindent
In this section, we shall give the second variation formula of $\mathcal{F}$-functional.
\begin{theorem}\label{theorem 5}
Let $X: M\rightarrow \mathbb{R}^{n+1}$ be a critical point of the functional $\mathcal{F}(s)=\mathcal{F}_{X_s,t_s}(s)$.  The second variation formula of $\mathcal{F}(s)$ for $X_0=0$ and $t_0=1$ is given by

\begin{equation*}
\aligned
 &\ \ \ \ (4\pi)^{\frac{n}{2}}\mathcal{F}^{''}(0)\\
 &=-\int_M fLf e^{-\frac{|X|^2}{2}}d\mu+\int_M \bigl(-|y|^2+\langle X,y\rangle ^2\bigl)e^{-\frac{|X|^2}{2}}d\mu\\
&\ \ \ +\int_M \biggl\{2\langle N,y\rangle +(n+1-|X|^2)\lambda h-2hH-2\lambda\langle X,y\rangle \biggl\}f e^{-\frac{|X|^2}{2}}d\mu\\
&\ \ \ +\int_M \biggl\{\lambda\langle N,y\rangle -(n+2)\langle X,y\rangle +\langle X,y\rangle |X|^2\biggl\}h e^{-\frac{|X|^2}{2}}d\mu\\
&\ \ \ +\int_M \biggl\{\frac{n^2+2n}{4}-\frac{n+2}{2}|X|^2+\frac{|X|^4}{4}+\frac{3\lambda}{4}(\lambda-H)\biggl\}h^2 e^{-\frac{|X|^2}{2}}d\mu,
\endaligned
\end{equation*}
where the operator $L$ is defined by
$$L=\mathcal{L}+S+1-\lambda^2.$$
\end{theorem}
\begin{proof}
\begin{equation*}
\aligned
&\ \ \mathcal{F}^{''}(s)\\
&=(4\pi t_s)^{-\frac{n}{2}}\int_M -(H_s+\langle \frac{X(s)-X_s}{t_s}, N(s)\rangle )f^{'}e^{-\frac{|X(s)-X_s|^2}{2t_s}}d\mu_s\\
&\ \ \ +(4\pi t_0)^{-\frac{n}{2}}\sqrt{\frac{t_0}{t_s}}\int_M \lambda f^{'}\langle N(s),N\rangle e^{-\frac{|X-X_0|^2}{2t_0}}d\mu\\
&\ \ \ +(4\pi t_s)^{-\frac{n}{2}}\int_M \frac{nh}{2t_s}(H_s+\langle \frac{X(s)-X_s}{t_s},N(s)\rangle fe^{-\frac{|X(s)-X_s|^2}{2t_s}}d\mu_s\\
&\ \ \ +(4\pi t_0)^{-\frac{n}{2}}\sqrt{\frac{t_0}{t_s}}\int_M -\frac{h}{2t_s}\lambda \langle N(s),N\rangle fe^{-\frac{|X-X_0|^2}{2t_0}}d\mu\\
&\ \ \ +(4\pi t_s)^{-\frac{n}{2}}\int_M (H_s+\langle \frac{X(s)-X_s}{t_s},N(s)\rangle )\times\\
&\ \ \ \ \ \ \ \ \ \ \ \ \ \ \ \ \ \ \ \ \ \ \ (\langle \frac{X(s)-X_s}{t_s},
   \frac{\partial X(s)}{\partial s}-\frac{\partial X_s}{\partial s}\rangle +H_sf)fe^{-\frac{|X(s)-X_s|^2}{2t_s}}d\mu_s\\
   & \ \   +(4\pi t_s)^{-\frac{n}{2}}\int_M -\biggl(\frac{dH_s}{ds}
   +\langle \frac{\frac{\partial X(s)}{\partial s}-\frac{\partial X_s}{\partial s}}{t_s},N(s)\rangle -\langle \frac{X(s)-X_s}{t_s^2},N(s)\rangle h\\
   &\ \ \ \ \ \ \ \ \ \ \ \ \ \ \ \ \ \ \ \ \ \ \ \ \   +\langle \frac{X(s)-X_s}{t_s},\frac{dN(s)}{ds}\rangle \biggl)fe^{-\frac{|X(s)-X_s|^2}{2t_s}}d\mu_s\\
&  \ \   +(4\pi t_0)^{-\frac{n}{2}}\sqrt{\frac{t_0}{t_s}}\int_M \lambda f\langle \frac{dN(s)}{ds},N\rangle e^{-\frac{|X-X_0|^2}{2t_0}}d\mu\\
& \ \   +(4\pi t_s)^{-\frac{n}{2}}\int_M\langle \frac{X(s)-X_s}{t_s},y^{'}\rangle e^{-\frac{|X(s)-X_s|^2}{2t_s}}d\mu_s\\
& \ \   +(4\pi t_0)^{-\frac{n}{2}}\sqrt{\frac{t_0}{t_s}}\int_M -\lambda\langle N, y^{'}\rangle e^{-\frac{|X-X_0|^2}{2t_0}}d\mu\\
&  \ \ +(4\pi t_s)^{-\frac{n}{2}}(-\frac{nh}{2t_s})\int_M\langle \frac{X(s)-X_s}{t_s},y\rangle e^{-\frac{|X(s)-X_s|^2}{2t_s}}d\mu_s\\
& \ \   +(4\pi t_0)^{-\frac{n}{2}}\sqrt{\frac{t_0}{t_s}}(-\frac{h}{2t_s})\int_M -\lambda\langle N,y\rangle e^{-\frac{|X-X_0|^2}{2t_0}}d\mu\\
&  \ \  +(4\pi t_s)^{-\frac{n}{2}}\int_M (\langle \frac{\frac{\partial X(s)}{\partial s}-\frac{\partial X_s}{\partial s}}{t_s},y\rangle
  -\langle \frac{X(s)-X_s}{t_s^2},y\rangle h)e^{-\frac{|X(s)-X_s|^2}{2t_s}}d\mu_s\\
& \ \ +(4\pi t_s)^{-\frac{n}{2}}\int_M \langle \frac{X(s)-X_s}{t_s},y\rangle \biggl(-\langle \frac{X(s)-X_s}{t_s},
  \frac{\partial X(s)}{\partial s}-\frac{\partial X_s}{\partial s}\rangle\
-H_sf\biggl)e^{-\frac{|X(s)-X_s|^2}{2t_s}}d\mu_s\\
& \ \   +(4\pi t_s)^{-\frac{n}{2}}\int_M(-\frac{n}{2t_s}+\frac{|X(s)-X_s|^2}{2t_s^2})h^{'}e^{-\frac{|X(s)-X_s|^2}{2t_s}}d\mu_s\\
&  \ \ +(4\pi t_0)^{-\frac{n}{2}}\sqrt{\frac{t_0}{t_s}}\int_M -\frac{h^{'}\lambda}{2t_s}\langle X(s)-X_s,N\rangle e^{-\frac{|X-X_0|^2}{2t_0}}d\mu
\\
\endaligned
\end{equation*}
\begin{equation*}
\aligned
&\ \ \ +(4\pi t_s)^{-\frac{n}{2}}(-\frac{nh}{2t_s})\int_M (-\frac{n}{2t_s}+\frac{|X(s)-X_s|^2}{2t_s^2})he^{-\frac{|X(s)-X_s|^2}{2t_s}}d\mu_s\ \ \ \ \ \ \ \ \ \\
&\ \ \ +(4\pi t_0)^{-\frac{n}{2}}\sqrt{\frac{t_0}{t_s}}(-\frac{h}{2t_s})
   \int_M -\frac{h}{2t_s}\lambda \langle X(s)-X_s, N\rangle e^{-\frac{|X-X_0|^2}{2t_0}}d\mu
   \\
&\ \ \ +(4\pi t_s)^{-\frac{n}{2}}\int_M (\frac{nh}{2t_s^2}-\frac{|X(s)-X_s|^2}{t_s^3}h
+\frac{\langle X(s)-X_s,\frac{\partial X(s)}{\partial s}-\frac{\partial X_s}{\partial s}\rangle }{t_s^2})\times\\
&\ \ \ \ \ \ \ \ \ \ \ \ \ \ \ \ \ \ \ \ \ \ \ \ \ he^{-\frac{|X(s)-X_s|^2}{2t_s}}d\mu_s\\
&\ \ \ +(4\pi t_0)^{-\frac{n}{2}}\sqrt{\frac{t_0}{t_s}}\int_M (\frac{h}{2t_s^2}\langle X(s)-X_s,N\rangle \lambda h\\
&\ \ \ \ \ \ \ \ \ \ \ \ \ \ \ \ \ \ \ \ \ \ \ \ \ \ \ \ \
 -\frac{1}{2t_s}\langle \frac{\partial X(s)}{\partial s}-\frac{\partial X_s}{\partial s},N\rangle \lambda h)e^{-\frac{|X-X_0|^2}{2t_0}}d\mu\\
&\ \ \ +(4\pi t_s)^{-\frac{n}{2}}\int_M (-\frac{n}{2t_s}+\frac{|X(s)-X_s|^2}{2t_s^2})h(-H_sf\\
&\ \ \ \ \ \ \ \ \ \ \ \ \ \ \ \ \ \ \ \ \ \ \ \
  -\langle \frac{X(s)-X_s}{t_s},\frac{\partial X(s)}{\partial s}-\frac{\partial X_s}{\partial s}\rangle )e^{-\frac{|X(s)-X_s|^2}{2t_s}}d\mu_s\\
&\ \ \ +(4\pi t_s)^{-\frac{n}{2}}\int_M-(H_s+\langle \frac{X(s)-X_s}{t_s},N(s)\rangle )f\frac{|X(s)-X_s|^2}{2t_s^2}h\\
&\ \ \ \ \ \ \ \ \ \ \ \ \ \ \ \ \ \ \ \ \ \ \ \ \ \ \times e^{-\frac{|X(s)-X_s|^2}{2t_s}}d\mu_s\\
&\ \ \ +(4\pi t_s)^{-\frac{n}{2}}\int_M \langle \frac{X(s)-X_s}{t_s},y\rangle \frac{|X(s)-X_s|^2}{2t_s^2}he^{-\frac{|X(s)-X_s|^2}{2t_s}}d\mu_s\\
&\ \ \ +(4\pi t_s)^{-\frac{n}{2}}\int_M (-\frac{n}{2t_s}+\frac{|X(s)-X_s|^2}{2t_s^2})h
 \frac{|X(s)-X_s|^2}{2t_s^2}he^{-\frac{|X(s)-X_s|^2}{2t_s}}d\mu_s.
\endaligned
\end{equation*}

Since $X: M\rightarrow \mathbb{R}^{n+1}$ is a critical point, we get

\begin{equation*}
H+\langle \frac{X-X_0}{t_0},N\rangle =\lambda,
\end{equation*}

\begin{equation*}
\int_M (n+\lambda\langle X-X_0,N\rangle -\frac{|X-X_0|^2}{t_0})e^{-\frac{|X-X_0|^2}{2t_0}}d\mu=0,
\end{equation*}

\begin{equation*}
\int_M (\lambda\langle N,a\rangle -\langle \frac{X-X_0}{t_0},a\rangle )e^{-\frac{|X-X_0|^2}{2t_0}}d\mu=0.
\end{equation*}
On the other hand,
\begin{equation*}
H^{'}=\Delta f+Sf,\ \ N^{'}=-\nabla f.
\end{equation*}
Using of the above equations and letting $s=0$, we obtain
\begin{equation*}
\aligned
&\ \ \ \ (4\pi t_0)^{\frac{n}{2}}\mathcal{F}^{''}(0)\\
 &=\int_M -fLf e^{-\frac{|X-X_0|^2}{2t_0}}d\mu\\
&\ \ \ +\int_M (\frac{2}{t_0}\langle N,y\rangle +\frac{2h}{t_0}\langle \frac{X-X_0}{t_0},N\rangle +\frac{n-1}{t_0}\lambda h\\
 &\ \ \ \ \ \ \ \ \ \ \ \ -\frac{|X-X_0|^2}{t_0^2}\lambda h-2\lambda\langle \frac{X-X_0}{t_0},y\rangle )fe^{-\frac{|X-X_0|^2}{2t_0}}d\mu\\
&\ \ \ +\int_M (-\frac{n+2}{t_0}\langle \frac{X-X_0}{t_0},y\rangle +\frac{\lambda}{t_0}\langle N,y\rangle \\
&\ \ \ \ \ \ \ \ \ \ \ \ +\langle \frac{X-X_0}{t_0},y\rangle
  \frac{|X-X_0|^2}{t_0^2})he^{-\frac{|X-X_0|^2}{2t_0}}d\mu\\
&\ \ \ +\int_M (\frac{n^2}{4t_0^2}+\frac{n}{2t_0^2}-\frac{n+2}{2t_0^3}|X-X_0|^2+\frac{|X-X_0|^4}{4t_0^4}\\
 &\ \ \ \ \ \ \ \ \ \ +\frac{3\lambda}{4t_0}\langle \frac{X-X_0}{t_0},N\rangle )h^2e^{-\frac{|X-X_0|^2}{2t_0}}d\mu\\
&\ \ \ +\int_M (-\frac{1}{t_0}\langle y,y\rangle +\langle \frac{X-X_0}{t_0},y\rangle ^2)e^{-\frac{|X-X_0|^2}{2t_0}}d\mu,
\endaligned
\end{equation*}
where the operator $L$ is defined by $L=\Delta+S+\frac{1}{t_0}-\langle \frac{X-X_0}{t_0},\nabla\rangle -\lambda^2$.
When $t_0=1$, $X_0=0$, then $L=\mathcal{L}+S+1-\lambda^2$.
\begin{equation*}
\aligned
&\ \ \ \ (4\pi)^{\frac{n}{2}}\mathcal{F}^{''}(0)\\
&=\int_M -fLf e^{-\frac{|X-|^2}{2}}d\mu\\
&\ \ \ +\int_M (2\langle N,y\rangle +2\lambda h+(n-1)\lambda h-2hH\\
&\ \ \ \ \ \ \ \ \ \ \ -|X|^2\lambda h-2\lambda\langle X,y\rangle )f e^{-\frac{|X|^2}{2}}d\mu\\
&\ \ \ +\int_M (\lambda\langle N,y\rangle -(n+2)\langle X,y\rangle +\langle X,y\rangle |X|^2)h e^{-\frac{|X|^2}{2}}d\mu\\
&\ \ \ +\int_M (\frac{n^2+2n}{4}-\frac{n+2}{2}|X|^2+\frac{|X|^4}{4}+\frac{3\lambda}{4}\langle X,N\rangle )h^2 e^{-\frac{|X|^2}{2}}d\mu\\
&\ \ \ +\int_M -(|y|^2-\langle X,y\rangle ^2)e^{-\frac{|X|^2}{2}}d\mu\\
\endaligned
\end{equation*}

\begin{equation*}
\aligned
&=\int_M -fLf e^{-\frac{|X|^2}{2}}d\mu\\
&\ \ \ +\int_M [2\langle N,y\rangle +(n+1-|X|^2)\lambda h-2hH-2\lambda\langle X,y\rangle ]f e^{-\frac{|X-|^2}{2}}d\mu\\
&\ \ \ +\int_M (\lambda\langle N,y\rangle -(n+2)\langle X,y\rangle +\langle X,y\rangle |X|^2)h e^{-\frac{|X|^2}{2}}d\mu\\
&\ \ \ +\int_M (\frac{n^2+2n}{4}-\frac{n+2}{2}|X|^2+\frac{|X|^4}{4}+\frac{3\lambda}{4}(\lambda-H))h^2 e^{-\frac{|X|^2}{2}}d\mu\\
&\ \ \ +\int_M (-|y|^2+\langle X,y\rangle ^2)e^{-\frac{|X|^2}{2}}d\mu.
\endaligned
\end{equation*}
\end{proof}

\begin{definition}
One calls that a critical point $X: M\rightarrow \mathbb{R}^{n+1}$ of the $\mathcal{F}$-functional $\mathcal{F}_{X_s,t_s}(s)$
is $\mathcal{F}$-stable if, for every
normal variation $fN$, there exist variations of $X_0$ and $t_0$ such that  $\mathcal{F^{\prime\prime}}_{X_0,t_0}(0)\geq 0$;

\noindent
One calls that a critical point $X: M\rightarrow \mathbb{R}^{n+1}$ of the $\mathcal{F}$-functional $\mathcal{F}_{X_s,t_s}(s)$
is $\mathcal{F}$-unstable if there exist a
normal variation $fN$ such that for all variations of $X_0$ and $t_0$,   $\mathcal{F^{\prime\prime}}_{X_0,t_0}(0)< 0$.
\end{definition}

\begin{theorem}\label{theorem 6.2}
If  $r\leq \sqrt{n}$ or  $r>\sqrt{n+1}$, the $n$-dimensional round sphere $X: {S}^n(r)\rightarrow \mathbb{R}^{n+1}$
is  $\mathcal{F}$-stable;
If  $ \sqrt{n}< r\leq\sqrt{n+1}$,  the $n$-dimensional round sphere $X: {S}^n(r)\rightarrow \mathbb{R}^{n+1}$
is  $\mathcal{F}$-unstable.
\end{theorem}
\begin{proof}
For the sphere ${S}^n(r)$, we have
$$
X=-rN,\ \ H=\frac{n}{r},\ \ S=\frac{H^2}{n}=\frac{n}{r^2},\ \ \lambda=H-r=\frac{n}{r}-r
$$
and
\begin{equation}
L f=\mathcal{L}f+(S+1-\lambda^2)f=\Delta f+(\frac{n}{r^2}+1-\lambda^2)f.
\end{equation}
Since we know that  eigenvalues $\mu_k$ of $\Delta$ on the sphere ${S}^n(r)$ are given by
\begin{equation}
\mu_k=\frac{k^2+(n-1)k}{r^2},
\end{equation}
and  constant functions are eigenfunctions corresponding to eigenvalue $\mu_0=0$.
For any constant vector $z\in \mathbb{R}^{n+1}$, we get
\begin{equation}
-\Delta\langle z,N\rangle =\Delta\langle z,\frac{X}{r}\rangle =\langle z,\frac{1}{r}HN\rangle =\frac{n}{r^2}\langle z,N\rangle ,
\end{equation}
that is, $\langle z,N\rangle $ is an eigenfunction of $\Delta$ corresponding to the first eigenvalue $\mu_1=\frac{n}{r^2}$.
Hence, for any  normal variation with the variation vector field $fN$,
we can choose a real number $a\in \mathbb{R}$ and a constant vector $z\in \mathbb{R}^{n+1}$ such that
\begin{equation}
f=f_0+a+\langle z,N\rangle ,
\end{equation}
and  $f_0$ is in the space spanned by all eigenfunctions corresponding to eigenvalues $\mu_k$ $(k\geq2)$ of $\Delta$ on ${S}^n(r)$.
Using the lemma \ref{lemma 10}, we get
\begin{equation}\label{eq:500}
\aligned
&\ \ \ \ (4\pi)^{\frac{n}{2}}e^{\frac{r^2}{2}}\mathcal{F}^{''}(0)\\
&=\int_{S^n(r)} -(f_0+a+\langle z,N\rangle )L(f_0+a+\langle z,N\rangle ) d\mu\\
&\ \ \ +\int_{S^n(r)} [2\langle N,y\rangle +(n+1-r^2)\lambda h-2\frac{n}{r}h+2\lambda\langle rN,y\rangle ] (f_0+a+\langle z,N\rangle ) d\mu\\
&\ \ \ +\int_{S^n(r)} (-r)\langle N,y\rangle (r^2-n-1)h d\mu\\
&\ \ \ +\int_{S^n(r)} (\frac{n^2+2n}{4}-\frac{n+2}{2}r^2+\frac{r^4}{4}+\frac{3}{4}r^2-\frac{3}{4}n)h^2 d\mu\\
&\ \ \ +\int_{S^n(r)} (-|y|^2+\langle X,y\rangle ^2)d\mu\\
&\geq\int_{S^n(r)} \biggl\{(\frac{n+2}{r^2}-1+\lambda^2)f_0^2-(\frac{n}{r^2}+1-\lambda^2)a^2+(\lambda^2-1)\langle z,N\rangle ^2\biggl\}d\mu\\
&\ \ \ +\int_{S^n(r)}\biggl\{2(1+\lambda r)\langle N,y\rangle \langle N,z\rangle +[(n+1-r^2)\lambda-2\frac{n}{r}]ah\biggl\}d\mu\\
&\ \ \ +\int_{S^n(r)}\frac{1}{4}[r^4-(2n+1)r^2+n(n-1)]h^2d\mu\\
&\ \ \ +\int_{S^n(r)}(-|y|^2+\langle X,y\rangle ^2)d\mu.
\endaligned
\end{equation}
From the lemma \ref{lemma 10}, we have
\begin{equation}\label{eq:501}
\int_{S^n(r)}(-|y|^2+\langle X,y\rangle ^2)d\mu=-\int_{S^n(r)} (1+\lambda r)\langle N,y\rangle ^2d\mu.
\end{equation}
Putting \eqref{eq:501} and $\lambda=\frac{n}{r}-r$ into \eqref{eq:500}, we obtain

\begin{equation}\label{eq:502}
\aligned
&\ \ \ \ (4\pi)^{\frac{n}{2}}e^{\frac{r^2}{2}}\mathcal{F}^{''}(0)\\
&\geq\int_{S^n(r)} \frac{1}{r^2}\biggl\{(r^2-n-\frac{1}{2})^2+\frac{7}{4}\biggl\}f_0^2d\mu\\
&\ \ \ +\int_{S^n(r)}[r^4-(2n+1)r^2+n(n-1)](\frac{a}{r}+\frac{h}{2})^2d\mu\\
&\ \ \ +\int_{S^n(r)}\frac{1}{r^2}[r^4-(2n+1)r^2+n^2]\langle z,N\rangle ^2d\mu\\
&\ \ \ +\int_{S^n(r)} 2(1+n-r^2)\langle N,y\rangle \langle N,z\rangle d\mu\\
&\ \ \ +\int_{S^n(r)} -(1+n-r^2)\langle N,y\rangle ^2d\mu.
\endaligned
\end{equation}
 If we choose $h=-\frac{2a}{r}$, then we have

\begin{equation}\label{eq:503}
\aligned
&\ \ \ \ (4\pi)^{\frac{n}{2}}e^{\frac{r^2}{2}}\mathcal{F}^{''}(0)\\
&\geq\int_{S^n(r)} \frac{1}{r^2}\biggl\{(r^2-n-\frac{1}{2})^2+\frac{7}{4}\biggl\}f_0^2d\mu\\
&\ \ \ +\int_{S^n(r)} (\lambda^2-1)\langle z,N\rangle ^2d\mu\\
&\ \ \ +\int_{S^n(r)} 2(1+\lambda r)\langle N,y\rangle \langle N,z\rangle d\mu\\
&\ \ \ +\int_{S^n(r)} -(1+\lambda r)\langle N,y\rangle ^2d\mu.
\endaligned
\end{equation}
Let $y=kz$, then we have

\begin{equation}\label{eq:504}
\aligned
&\ \ \ \ (4\pi)^{\frac{n}{2}}e^{\frac{r^2}{2}}\mathcal{F}^{''}(0)\\
&\geq\int_{S^n(r)} \frac{1}{r^2}\biggl\{(r^2-n-\frac{1}{2})^2+\frac{7}{4}\biggl\}f_0^2d\mu\\
&+\int_{S^n(r)} \biggl\{\lambda^2-1+2(1+\lambda r)k-(1+\lambda r)k^2\biggl\}\langle z,N\rangle ^2d\mu\\
&=\int_{S^n(r)} \frac{1}{r^2}\biggl\{(r^2-n-\frac{1}{2})^2+\frac{7}{4}\biggl\}f_0^2d\mu\\
&+\int_{S^n(r)} \biggl\{\lambda^2+\lambda r-(1+\lambda r)(1-k)^2\biggl\}\langle z,N\rangle ^2d\mu.
\endaligned
\end{equation}
We next consider three cases:

\vskip 3mm
{\bf Case 1: $r\leq \sqrt{n}$}
\vskip3mm

\noindent
In this case, $\lambda\geq 0$. Taking $k=1$, then we get
$$\mathcal{F}^{''}(0)\geq 0.$$

\vskip3mm
{\bf Case 2: $r\geq\frac{1+\sqrt{1+4n}}{2}$}.
\vskip3mm

\noindent
In this case, $\lambda\leq -1$.  Taking $k=2$, we can get

$$\mathcal{F}^{''}(0)\geq 0.$$

\vskip3mm
{\bf Case 3: $\sqrt{n+1}< r< \frac{1+\sqrt{1+4n}}{2}$}.
\vskip3mm

\noindent
In this case, $-1< \lambda<0$, $1+\lambda r<0$, we can take $k$ such that
$(1-k)^2\geq\frac{\lambda(\lambda+r)}{1+\lambda r}$, then we have

$$\mathcal{F}^{''}(0)\geq 0.
$$
Thus, if $r\leq \sqrt{n}$ or  $r>\sqrt{n+1}$, the $n$-dimensional round sphere $X: {S}^n(r)\rightarrow \mathbb{R}^{n+1}$
is  $\mathcal{F}$-stable;

\vskip3mm
\noindent
If  $ \sqrt{n}< r\leq\sqrt{n+1}$,  the $n$-dimensional round sphere $X: {S}^n(r)\rightarrow \mathbb{R}^{n+1}$
is  $\mathcal{F}$-unstable. In fact,
in this case, $-1< \lambda< 0$, $1+\lambda r\geq0$.  We can choose $f$ such that $f_0=0$, then we have
\begin{equation}
\aligned
(4\pi)^{\frac{n}{2}}e^{\frac{r^2}{2}}\mathcal{F}^{''}(0)
&=\int_{S^n(r)} (\lambda^2-1)\langle z,N\rangle ^2d\mu\\
&\ \ \ +\int_{S^n(r)} 2(1+\lambda r)\langle N,y\rangle \langle N,z\rangle d\mu\\
&\ \ \ +\int_{S^n(r)}-(1+\lambda r)\langle N,y\rangle ^2d\mu\\
&=(\lambda^2+\lambda r)\int_{S^n(r)}\langle z,N\rangle ^2d\mu\\
&\ \ \ -(1+\lambda r)\int_{S^n(r)} (\langle z,N\rangle -\langle y,N\rangle )^2d\mu\\
&< 0.
\endaligned
\end{equation}
 This completes the proof of the theorem \ref{theorem 6.2}.
\end{proof}

\noindent
According to our theorem 6.2, we would like to propose the following:

\vskip2mm
\noindent
{\bf Problem 6.1}.  Is it possible to prove that
spheres $S^n(r)$ with $r\leq \sqrt n$ or $r>\sqrt {n+1}$ are  the only $\mathcal F$-stable
compact $\lambda$-hypersurfaces?
\begin{remark}   Colding and Minicozzi \cite{[CZ]} have proved that the sphere $S^n(\sqrt n)$ is
the only $\mathcal F$-stable compact self-shrinkers. In order to prove this result,
the property that the mean curvature $H$ is an eigenfunction of $L$-operator plays a very important role.
But for $\lambda$-hypersurfaces, the mean curvature $H$ is not an eigenfunction of $L$-operator in general.
\end{remark}

\vskip 5mm

\section{The weak stability of the weighted area functional for  weighted volume-preserving variations}

\noindent
Define
\begin{equation}
\mathcal{T}(s)=(4\pi t_s)^{-\frac{n}{2}}\int_M e^{-\frac{|X(s)-X_s|^2}{2t_s}}d\mu_s.
\end{equation}
We compute the first and the second variation formulas of the general $\mathcal{T}$-functional for weighted volume-preserving variations.
By a direct calculation, we have
\begin{equation*}
\aligned
&\ \ \ \ \mathcal{T}^{'}(s)\\
&=(4\pi t_s)^{-\frac{n}{2}}\int_M -(H_s+\langle \frac{X(s)-X_s}{t_s}, N(s)\rangle )fe^{-\frac{|X(s)-X_s|^2}{2t_s}}d\mu_s\\
&\ \ \ +(4\pi t_s)^{-\frac{n}{2}}\int_M \langle \frac{X(s)-X_s}{t_s},y\rangle e^{-\frac{|X(s)-X_s|^2}{2t_s}}d\mu_s\\
&\ \ \ +(4\pi t_s)^{-\frac{n}{2}}\int_M (-\frac{n}{2t_s}+\frac{|X(s)-X_s|^2}{2t_s^2})he^{-\frac{|X(s)-X_s|^2}{2t_s}}d\mu_s.
\endaligned
\end{equation*}

\begin{equation*}
\aligned
&\ \ \ \ \mathcal{T}^{''}(s)\\
&=(4\pi t_s)^{-\frac{n}{2}}\int_M -(H_s+\langle \frac{X(s)-X_s}{t_s}, N(s)\rangle )f^{'}e^{-\frac{|X(s)-X_s|^2}{2t_s}}d\mu_s\\
&\ \ \ +(4\pi t_s)^{-\frac{n}{2}}\int_M \frac{nh}{2t_s}(H_s+\langle \frac{X(s)-X_s}{t_s},N(s)\rangle fe^{-\frac{|X(s)-X_s|^2}{2t_s}}d\mu_s\\
&\ \ \ +(4\pi t_s)^{-\frac{n}{2}}\int_M (H_s+\langle \frac{X(s)-X_s}{t_s},N(s)\rangle )\times\\
&\ \ \ \ \ \ \ \ \ \ \ \ \ \ \ \ \ \ \ \ \ \ \ (\langle \frac{X(s)-X_s}{t_s},
   \frac{\partial X(s)}{\partial s}-\frac{\partial X_s}{\partial s}\rangle +H_sf)fe^{-\frac{|X(s)-X_s|^2}{2t_s}}d\mu_s\\
   &    \ \ \ \ +(4\pi t_s)^{-\frac{n}{2}}\int_M -\biggl(\frac{dH_s}{ds}
   +\langle \frac{\frac{\partial X(s)}{\partial s}-\frac{\partial X_s}{\partial s}}{t_s},N(s)\rangle -\langle \frac{X(s)-X_s}{t_s^2},N(s)\rangle h\\
   & \ \ \ \ \ \ \ \ \ \ \ \ \ \ \ \ \ \ \ \ \ \ \ \ +\langle \frac{X(s)-X_s}{t_s},\frac{dN(s)}{ds}\rangle \biggl)fe^{-\frac{|X(s)-X_s|^2}{2t_s}}d\mu_s\\
& \ \ \ \ +(4\pi t_s)^{-\frac{n}{2}}\int_M\langle \frac{X(s)-X_s}{t_s},y^{'}\rangle e^{-\frac{|X(s)-X_s|^2}{2t_s}}d\mu_s
\endaligned
\end{equation*}

\begin{equation*}
\aligned
&  \ \ +(4\pi t_s)^{-\frac{n}{2}}(-\frac{nh}{2t_s})\int_M\langle \frac{X(s)-X_s}{t_s},y\rangle e^{-\frac{|X(s)-X_s|^2}{2t_s}}d\mu_s\ \ \\
&  \ \ +(4\pi t_s)^{-\frac{n}{2}}\int_M (\langle \frac{\frac{\partial X(s)}{\partial s}-\frac{\partial X_s}{\partial s}}{t_s},y\rangle
  -\langle \frac{X(s)-X_s}{t_s^2},y\rangle h)e^{-\frac{|X(s)-X_s|^2}{2t_s}}d\mu_s
\endaligned
\end{equation*}

\begin{equation*}
\aligned
& +(4\pi t_s)^{-\frac{n}{2}}\int_M \langle \frac{X(s)-X_s}{t_s},y\rangle \biggl(-\langle \frac{X(s)-X_s}{t_s},
  \frac{\partial X(s)}{\partial s}-\frac{\partial X_s}{\partial s}\rangle\ \ \  \\
& \ \ \ \ \ \ \  \ \ \ \ \ \ \ \ \ \ \ \ \
-H_sf\biggl)e^{-\frac{|X(s)-X_s|^2}{2t_s}}d\mu_s\\
& +(4\pi t_s)^{-\frac{n}{2}}\int_M(-\frac{n}{2t_s}+\frac{|X(s)-X_s|^2}{2t_s^2})h^{'}e^{-\frac{|X(s)-X_s|^2}{2t_s}}d\mu_s\\
& +(4\pi t_s)^{-\frac{n}{2}}(-\frac{nh}{2t_s})\int_M (-\frac{n}{2t_s}+\frac{|X(s)-X_s|^2}{2t_s^2})he^{-\frac{|X(s)-X_s|^2}{2t_s}}d\mu_s
\endaligned
\end{equation*}

\begin{equation*}
\aligned
&\ \ \ +(4\pi t_s)^{-\frac{n}{2}}\int_M (\frac{nh}{2t_s^2}-\frac{|X(s)-X_s|^2}{t_s^3}h
+\frac{\langle X(s)-X_s,\frac{\partial X(s)}{\partial s}-\frac{\partial X_s}{\partial s}\rangle }{t_s^2})\times\\
&\ \ \ \ \ \ \ \ \ \ \ \ \ \ \ \ \ \ \ \ \ \ \ \ \ he^{-\frac{|X(s)-X_s|^2}{2t_s}}d\mu_s\\
&\ \ \ \ \ \ \ \ \ \ \ \ \ \ \ \ \ \ \ \ \ \ \ \ \ \ \ \ \
 -\frac{1}{2t_s}\langle \frac{\partial X(s)}{\partial s}-\frac{\partial X_s}{\partial s},N\rangle \lambda h)e^{-\frac{|X-X_0|^2}{2t_0}}d\mu\\
&\ \ \ +(4\pi t_s)^{-\frac{n}{2}}\int_M (-\frac{n}{2t_s}+\frac{|X(s)-X_s|^2}{2t_s^2})h(-H_sf\\
&\ \ \ \ \ \ \ \ \ \ \ \ \ \ \ \ \ \ \ \ \ \ \ \
  -\langle \frac{X(s)-X_s}{t_s},\frac{\partial X(s)}{\partial s}-\frac{\partial X_s}{\partial s}\rangle )e^{-\frac{|X(s)-X_s|^2}{2t_s}}d\mu_s\\
\endaligned
\end{equation*}

\begin{equation*}
\aligned
&\ \ +(4\pi t_s)^{-\frac{n}{2}}\int_M-(H_s+\langle \frac{X(s)-X_s}{t_s},N(s)\rangle )f\frac{|X(s)-X_s|^2}{2t_s^2}h\\
&\ \ \ \ \ \ \ \ \ \ \ \ \ \ \ \ \ \ \ \ \ \ \ \ \ \times e^{-\frac{|X(s)-X_s|^2}{2t_s}}d\mu_s\\
&\ \  +(4\pi t_s)^{-\frac{n}{2}}\int_M \langle \frac{X(s)-X_s}{t_s},y\rangle \frac{|X(s)-X_s|^2}{2t_s^2}he^{-\frac{|X(s)-X_s|^2}{2t_s}}d\mu_s\\
&\ \  +(4\pi t_s)^{-\frac{n}{2}}\int_M (-\frac{n}{2t_s}+\frac{|X(s)-X_s|^2}{2t_s^2})h
 \frac{|X(s)-X_s|^2}{2t_s^2}he^{-\frac{|X(s)-X_s|^2}{2t_s}}d\mu_s.
\endaligned
\end{equation*}

\begin{lemma}
$$\int_Mf^{'}(0)e^{-\frac{|X-X_0|^2}{2t_0}}d\mu=0.$$
\end{lemma}
\begin{proof}
Since $V(t)=\int_M\langle X(t)-X_0,N\rangle e^{-\frac{|X-X_0|^2}{2t_0}}d\mu=V(0)$ for any $t$, we have
$$\int_Mf(t)\langle N(t),N\rangle e^{-\frac{|X-X_0|^2}{2t_0}}d\mu=0.
$$
Hence, we get
\begin{equation*}
\aligned
0&=\frac{d}{dt}|_{t=0}\int_Mf(t)\langle N(t),N\rangle e^{-\frac{|X-X_0|^2}{2t_0}}d\mu\\
 &=\int_M f^{'}(0)e^{-\frac{|X-X_0|^2}{2t_0}}d\mu.
\endaligned
\end{equation*}
\end{proof}

\noindent
Since $M$ is a critical point of $\mathcal{T}(s)$, we have
$$
H+\langle \frac{X-X_0}{t_0},N\rangle =\lambda.
$$
On the other hand, we have
\begin{equation}
H^{'}=\Delta f+Sf, \ \ N^{'}=-\nabla f.
\end{equation}
Then for $t_0=1$ and $X_0=0$, the second variation formula becomes
\begin{equation*}
\aligned
&\ \ \ \ (4\pi)^{\frac{n}{2}}\mathcal{T}^{''}(0)\\
&=\int_M\langle X,y^{'}\rangle e^{-\frac{|X-|^2}{2}}d\mu+\int_M(\frac{|X|^2}{2}-\frac{n}{2})h^{'}e^{-\frac{|X|^2}{2}}d\mu\\
&\ \ \ +\int_M -f\bigl(\mathcal{L}f+(S+1-\lambda^2)f\bigl) e^{-\frac{|X|^2}{2}}d\mu\\
&\ \ \ +\int_M \biggl(2\langle N,y\rangle +(n-|X|^2)\lambda h+2\langle X,N\rangle h\\
&\ \ \ \ \ \ \ \ \ \ \ \ \ +2\langle N,y\rangle -2\lambda\langle X,y\rangle \biggl)f e^{-\frac{|X|^2}{2}}d\mu\\
&\ \ \ +\int_M (-(n+2)\langle X,y\rangle +\langle X,y\rangle |X|^2)h e^{-\frac{|X-|^2}{2}}d\mu\\
\endaligned
\end{equation*}

\begin{equation*}
\aligned
&\ \ \ +\int_M (\frac{n^2+2n}{4}-\frac{n+2}{2}|X|^2+\frac{|X|^4}{4})h^2 e^{-\frac{|X|^2}{2}}d\mu\\
&\ \ \ +\int_M (-|y|^2+\langle X,y\rangle ^2)e^{-\frac{|X|^2}{2}}d\mu.
\endaligned
\end{equation*}

\begin{theorem}\label{theorem 7.1}
Let $X: M\rightarrow \mathbb{R}^{n+1}$ be a critical point of the functional $\mathcal{T}(s)$ for the weighted volume-preserving variations with fixed $X_0=0$ and $t_0=1$.
The second variation formula of $\mathcal{T}(s)$  is given by
\begin{equation}
\aligned
&\ \ \ \ (4\pi)^{\frac{n}{2}}\mathcal{T}^{''}(0)=\int_M -f\bigl(\mathcal{L}f+(S+1-\lambda^2)f\bigl) e^{-\frac{|X|^2}{2}}d\mu.
\endaligned
\end{equation}
\end{theorem}

\begin{definition}
A critical point $X: M\rightarrow \mathbb{R}^{n+1}$ of the  functional $\mathcal{T}(s)$
is  called weakly stable if, for any  weighted volume-preserving normal variation,
  $\mathcal{T^{\prime\prime}}(0)\geq 0$;

\noindent
A critical point $X: M\rightarrow \mathbb{R}^{n+1}$ of the  functional $\mathcal{T}(s)$
is  called weakly unstable if there exists  a weighted volume-preserving normal variation,
  such that $\mathcal{T^{\prime\prime}}(0)< 0$.
\end{definition}

\begin{theorem}\label{theorem 7.2}
If  $r\leq \frac{-1+\sqrt{1+4n}}{2}$ or  $r\geq \frac{1+\sqrt{1+4n}}{2}$, the $n$-dimensional round sphere $X: {S}^n(r)\rightarrow \mathbb{R}^{n+1}$
is   weakly stable;
If  $ \frac{-1+\sqrt{1+4n}}{2}< r < \frac{1+\sqrt{1+4n}}{2}$,  the $n$-dimensional round sphere $X: {S}^n(r)\rightarrow \mathbb{R}^{n+1}$
is  weakly unstable.
\end{theorem}
\begin{proof}
For the sphere ${S}^n(r)$, we have
$$
X=-rN,\ \ H=\frac{n}{r},\ \ S=\frac{n}{r^2},\ \ \lambda=H-r=\frac{n}{r}-r
$$
and
\begin{equation}
Lf=\mathcal{L}f+(S+1-\lambda^2)f=\Delta f+(\frac{n}{r^2}+1-\lambda^2)f.
\end{equation}
Since we know that  eigenvalues $\mu_k$ of $\Delta$ on the sphere ${S}^n(r)$ are given by
\begin{equation}
\mu_k=\frac{k^2+(n-1)k}{r^2},
\end{equation}
and  constant functions are eigenfunctions corresponding to eigenvalue $\mu_0=0$. For any constant vector $z\in \mathbb{R}^{n+1}$, we get
\begin{equation}
-\Delta\langle z,N\rangle =\frac{n}{r^2}\langle z,N\rangle ,
\end{equation}
that is, $\langle z,N\rangle $ is an eigenfunction of $\Delta$ corresponding to the first eigenvalue $\mu_1=\frac{n}{r^2}$.
Hence, for any weighted volume-preserving normal variation with the variation vector field $fN$ satisfying
$$
\int_{S^n(r)} fe^{-\frac{r^2}{2}}d\mu=0,
$$
we can choose  a constant vector $z\in \mathbb{R}^{n+1}$ such that

\begin{equation}
f=f_0+\langle z,N\rangle ,
\end{equation}
and  $f_0$ is in the space spanned by all eigenfunctions corresponding to eigenvalues $\mu_k$ $(k\geq2)$ of $\Delta$ on ${S}^n(r)$.
By making use of the theorem  \ref{theorem 7.1}, we have
\begin{equation}\label{eq:700}
\aligned
&\ \ \ \ (4\pi)^{\frac{n}{2}}e^{\frac{r^2}{2}}\mathcal{T}^{''}(0)\\
&=\int_{S^n(r)} -(f_0+\langle z,N\rangle )L(f_0+\langle z,N\rangle ) d\mu\\
&\geq\int_{S^n(r)} \biggl\{(\frac{n+2}{r^2}-1+\lambda^2)f_0^2+(\lambda^2-1)\langle z,N\rangle ^2\biggl\}d\mu.
\endaligned
\end{equation}
According to  $\lambda=\frac{n}{r}-r$, we obtain

\begin{equation*}
\aligned
&(4\pi)^{\frac{n}{2}}e^{\frac{r^2}{2}}\mathcal{T}^{''}(0)\\
&\geq\int_{S^n(r)} \frac{1}{r^2}\bigl\{(r^2-n-\frac{1}{2})^2+\frac{7}{4}\bigl\}f_0^2d\mu
 +\int_{S^n(r)} (\frac nr-r-1)(\frac nr-r+1)\langle z,N\rangle ^2d\mu\geq 0
\endaligned
\end{equation*}
if
$$
r\leq \dfrac{-1+\sqrt{4n+1}}2  \ \ \ {\rm or} \ \ \ r\geq \dfrac{1+\sqrt{4n+1}}2 .
$$
Thus,  the $n$-dimensional round sphere $X: {S}^n(r)\rightarrow \mathbb{R}^{n+1}$
is  weakly stable.

\noindent
If
$$
 \dfrac{-1+\sqrt{4n+1}}2 < r<\dfrac{1+\sqrt{4n+1}}2,
$$
 choosing $f=<z, N>$,
 we have
 $$
\int_{S^n(r)} fe^{-\frac{r^2}{2}}d\mu=0.
$$
Hence,  there exists a weighted volume-preserving normal variation
with the variation vector filed $fN$ such that
\begin{equation*}
\aligned
(4\pi)^{\frac{n}{2}}e^{\frac{r^2}{2}}\mathcal{T}^{''}(0)
=\int_{S^n(r)} (\frac nr-r-1)(\frac nr-r+1)\langle z,N\rangle ^2d\mu< 0.
\endaligned
\end{equation*}
Thus, the $n$-dimensional round sphere $X: {S}^n(r)\rightarrow \mathbb{R}^{n+1}$
is  weakly unstable.
It finishes the proof.
\end{proof}

\begin{remark}
From the theorem \ref{theorem 6.2} and theorem \ref{theorem 7.2}, we know the $\mathcal{F}$-stability and
the weak stability are different. The $\mathcal{F}$-stability is a weaker notation than
the weak stability.
\end{remark}
\begin{remark}
Is it possible to prove that
spheres $S^n(r)$ with $r\leq \frac{-1+\sqrt{1+4n}}{2}$ or  $r\geq \frac{1+\sqrt{1+4n}}{2}$ are  the only  weak stable
compact $\lambda$-hypersurfaces?
\end{remark}
\vskip 5mm

\section{Complete and  non-compact $\lambda$-hypersurfaces}

\noindent
In this section, we will give  a  classification of complete and non-compact $\lambda$-hypersurfaces.

\begin{theorem}\label{theorem 10}
$S^k(r)\times \mathbb{R}^{n-k}$, $0\leq k\leq n$, are the only  complete embedded $\lambda$-hypersurfaces with polynomial area growth
in $\mathbb{R}^{n+1}$ if  $H-\lambda\geq 0$ and $\lambda(f_3(H-\lambda)-S)\geq0$.
\end{theorem}

\begin{remark}  The assumption  $\lambda(f_3(H-\lambda)-S)\geq 0$ in the theorem \ref{theorem 10} is satisfied
for self-shrinkers of the mean curvature flow, automatically and  the assumption is essential.
In fact, $\Gamma \times \mathbb{R}^{n-1}$ are counterexamples, which satisfy  $H-\lambda>0$, 
where $\Gamma$ are compact embedded $\lambda$-curves other than the circle (see Remark 2.2). 
\end{remark}

\noindent
At first, we prepare the following lemmas and propositions.

\begin{lemma}\label{lemma 001}
Let $X: M\rightarrow \mathbb{R}^{n+1}$ be an $n$-dimensional immersed hypersurface in the $(n+1)$-dimensional
Euclidean space $\mathbb{R}^{n+1}$. At any point $p\in M$, we have
\begin{equation}
|\nabla \sqrt{S}|^2\leq\sum_{i,k} h_{iik}^2\leq \sum_{i,j,k}h_{ijk}^2,
\end{equation}

\begin{equation}
\frac{n+3}{n+1}|\nabla \sqrt{S}|^2\leq \sum_{i,j,k}h_{ijk}^2+\frac{2n}{n+1}|\nabla H|^2.
\end{equation}
\end{lemma}

\noindent
Its proof is standard. See Schoen, Simon and Yau \cite{[SSY]} and Colding and Minicozzi \cite{[CM]}.

\begin{proposition}\label{proposition 10}
Let $X: M\rightarrow \mathbb{R}^{n+1}$ be an $n$-dimensional complete $\lambda$-hypersurface with $H-\lambda>0$ and
$\lambda(f_3-\frac{S}{H-\lambda})\geq0$. If $\eta$ is a function with compact support, then
\begin{equation}
\int_M\eta^2(S+|\nabla\log (H-\lambda)|^2)e^{-\frac{|X|^2}{2}}d\mu\leq c(n,\lambda)\int_M(|\nabla\eta|^2
 +\eta^2)e^{-\frac{|X|^2}{2}}d\mu,
\end{equation}
where $c(n,\lambda)$ is constant depending on $n$ and $\lambda$.
\end{proposition}
\begin{proof}
Since $H-\lambda>0$, $\log (H-\lambda)$ is well-defined. Suppose $\eta$ is a function with compact support,
the lemma \ref{lemma 2} and the corollary \ref{corollary 1} give
\begin{equation} \label{eq:001-1}
\aligned
&\ \ \ \int_M \langle \nabla \eta^2, \nabla\log(H-\lambda)\rangle e^{-\frac{|X|^2}{2}}d\mu\\
&=-\int_M \eta^2(\mathcal{L}\log(H-\lambda))e^{-\frac{|X|^2}{2}}d\mu\\
&=\int_M \eta^2\biggl(S-1-\frac{\lambda}{H-\lambda}+|\nabla\log(H-\lambda)|^2\biggl)e^{-\frac{|X|^2}{2}}d\mu.
\endaligned
\end{equation}
Combining this with inequality:
\begin{equation}
\langle \nabla \eta^2, \nabla\log(H-\lambda)\rangle \leq \varepsilon |\nabla \eta|^2+\frac{1}{\varepsilon}\eta^2|\nabla\log(H-\lambda)|^2
\end{equation}
gives that
\begin{equation}\label{eq:002}
\aligned
&\ \ \ \int_M (\eta^2S+\eta^2(1-\frac{1}{\varepsilon})|\nabla\log(H-\lambda)|^2)e^{-\frac{|X|^2}{2}}d\mu\\
&\leq\int_M(\varepsilon|\nabla \eta|^2+\eta^2+\frac{\lambda}{H-\lambda}\eta^2)e^{-\frac{|X|^2}{2}}d\mu,
\endaligned
\end{equation}
for $\varepsilon> 0$.
Since
\begin{equation}\label{eq:001}
\frac{\lambda}{H-\lambda}\leq\frac{\lambda f_3}{S}\leq |\lambda|\sqrt{S}\leq |\lambda|(\frac{S}{2\delta}
+\frac{\delta}{2})
\end{equation}
for $\delta>0$, we have from \eqref{eq:002} and \eqref{eq:001}
\begin{equation}\label{eq:003}
\aligned
&\ \ \ \int_M \biggl\{(1-\frac{|\lambda|}{2\delta})\eta^2S+\eta^2(1-\frac{1}{\varepsilon})|\nabla\log(H-\lambda)|^2
\biggl\}e^{-\frac{|X|^2}{2}}d\mu\\
&\leq\int_M\biggl(\varepsilon|\nabla \eta|^2+\bigl(1+\frac{|\lambda|}{2}\delta\bigl)\eta^2\biggl)e^{-\frac{|X|^2}{2}}d\mu.
\endaligned
\end{equation}
By choosing $\varepsilon$, $\delta$ and constant $c(n,\lambda)$, we get
\begin{equation}
\int_M\eta^2(S+|\nabla\log (H-\lambda)|^2)e^{-\frac{|X|^2}{2}}d\mu\leq c(n,\lambda)\int_M(|\nabla\eta|^2
 +\eta^2)e^{-\frac{|X|^2}{2}}d\mu.
\end{equation}
\end{proof}

\begin{proposition}\label{proposition 11}
Let $X: M\rightarrow \mathbb{R}^{n+1}$ be an $n$-dimensional complete  $\lambda$-hypersurface with $H-\lambda>0$ and
$\lambda(f_3-\frac{S}{H-\lambda})\geq0$. If $M$ has polynomial area growth, then
\begin{equation}\label{eq:225-2}
\aligned
&\ \ \ \int_M\langle\nabla S,\nabla\log(H-\lambda)\rangle e^{-\frac{|X|^2}{2}}d\mu\\
&=-\int_MS\mathcal{L}\log(H-\lambda)e^{-\frac{|X|^2}{2}}d\mu\\
&=\int_MS\biggl(S-1-\frac{\lambda}{H-\lambda}+|\nabla\log(H-\lambda)|^2\biggl)e^{-\frac{|X|^2}{2}}d\mu,
\endaligned
\end{equation}
and
\begin{equation}\label{eq:225-3}
\aligned
&\ \ \ \int_M|\nabla\sqrt{S}|^2e^{-\frac{|X|^2}{2}}d\mu\\
&=-\int_M\sqrt{S}\mathcal{L}\sqrt{S}e^{-\frac{|X|^2}{2}}d\mu\\
&\leq\int_M(S^2-S-\lambda f_3)e^{-\frac{|X|^2}{2}}d\mu.
\endaligned
\end{equation}
\end{proposition}

\begin{proof}
Taking $\eta= \phi$ in \eqref{eq:001-1}, we have
\begin{equation}
\aligned
&\ \ \ \int_M \langle \nabla \phi^2, \nabla\log(H-\lambda)\rangle e^{-\frac{|X|^2}{2}}d\mu\\
&=-\int_M \phi^2(\mathcal{L}\log(H-\lambda))e^{-\frac{|X|^2}{2}}d\mu\\
&=\int_M \phi^2\biggl(S-1-\frac{\lambda}{H-\lambda}+|\nabla\log(H-\lambda)|^2\biggl)e^{-\frac{|X|^2}{2}}d\mu.
\endaligned
\end{equation}
Since
\begin{equation}
\langle \nabla \phi^2, \nabla\log(H-\lambda)\rangle \leq |\nabla \phi|^2+\phi^2|\nabla\log(H-\lambda)|^2,
\end{equation}
we derive
\begin{equation}
\int_M \phi^2Se^{-\frac{|X|^2}{2}}d\mu
\leq\int_M(|\nabla \phi|^2+\phi^2+\frac{\lambda}{H-\lambda}\phi^2)e^{-\frac{|X|^2}{2}}d\mu.
\end{equation}
Let $\phi=\eta\sqrt{S}$, where $\eta\geq0$ has a compact support, for $\alpha> 0$, we have
\begin{equation}\label{eq:224-1}
\aligned
&\ \ \ \int_M \eta^2S^2e^{-\frac{|X|^2}{2}}d\mu\\
&\leq\int_M\bigl\{\eta^2|\nabla\sqrt{S}|^2+2\eta\sqrt{S}|\nabla\eta||\nabla\sqrt{S}|\\
&\ \ \ +S|\nabla\eta|^2
+(1+\frac{\lambda}{H-\lambda})\eta^2S\bigl\}e^{-\frac{|X|^2}{2}}d\mu\\
&\leq\int_M(1+\alpha)\eta^2|\nabla\sqrt{S}|^2e^{-\frac{|X|^2}{2}}d\mu\\
&\ \ \ +\int_MS\bigl\{(1+\frac{1}{\alpha})|\nabla\eta|^2+
(1+\frac{\lambda}{H-\lambda})\eta^2\bigl\}e^{-\frac{|X|^2}{2}}d\mu.
\endaligned
\end{equation}
The lemma \ref{lemma 2} and lemma \ref{lemma 001} give the following inequality
\begin{equation}
\aligned
\mathcal{L}S&=2\sum_{i,j,k}h_{ijk}^2+2(1-S)S+2\lambda f_3\\
            &\geq\frac{2(n+3)}{n+1}|\nabla\sqrt{S}|^2-\frac{4n}{n+1}|\nabla H|^2+2S-2S^2+2\lambda f_3,
\endaligned
\end{equation}
Integrating this with $\frac{1}{2}\eta^2$ and using the lemma \ref{lemma 1}, we obtain
\begin{equation*}
\aligned
&\ \ \ -2\int_M\eta\sqrt{S}\langle \nabla\eta,\nabla\sqrt{S}\rangle e^{-\frac{|X|^2}{2}}d\mu\\
&\geq\int_M\biggl\{\eta^2\frac{(n+3)}{n+1}|\nabla\sqrt{S}|^2-\frac{2n}{n+1}\eta^2|\nabla H|^2+
S\eta^2-S^2\eta^2+\lambda f_3\eta^2\biggl)e^{-\frac{|X|^2}{2}}d\mu.
\endaligned
\end{equation*}
Since  $2ab\leq\epsilon a^2+\frac{b^2}{\epsilon}$ for $\epsilon>0$, we infer
\begin{equation}\label{eq:224-2}
\aligned
&\ \ \ \int_M\biggl\{\eta^2S^2+\frac{2n}{n+1}\eta^2|\nabla H|^2+\frac{1}{\epsilon}S|\nabla \eta|^2\biggl\}e^{-\frac{|X|^2}{2}}d\mu\\
&\geq\int_M\biggl\{(\frac{n+3}{n+1}-\epsilon)\eta^2|\nabla \sqrt{S}|^2+S\eta^2+\lambda f_3\eta^2\biggl\}e^{-\frac{|X|^2}{2}}d\mu.
\endaligned
\end{equation}
From  \eqref{eq:224-1}, \eqref{eq:224-2} and $\lambda\frac{S}{H-\lambda}\leq\lambda f_3$, by taking $\alpha$ and  $\epsilon$ such that  $\frac{1+\alpha}{\frac{n+3}{n+1}-\epsilon}> 0$, we have

\begin{equation*}
\aligned
 &\ \ \ \ \int_M \eta^2S^2e^{-\frac{|X|^2}{2}}d\mu\\
 &\leq \frac{1+\alpha}{\frac{n+3}{n+1}-\epsilon}\int_M \eta^2S^2e^{-\frac{|X|^2}{2}}d\mu
+\frac{2n}{n+1}\cdot\frac{1+\alpha}{\frac{n+3}{n+1}-\epsilon}\int_M\eta^2|\nabla H|^2e^{-\frac{|X|^2}{2}}d\mu\\
 &\ \ \ +\int_M\biggl[\frac{1+\alpha}{\frac{n+3}{n+1}-\epsilon}\bigl(\frac{1}{\epsilon}|\nabla \eta|^2-\eta^2\bigl)
  +(1+\frac{1}{\alpha})|\nabla \eta|^2+(1+\frac{\lambda}{H-\lambda})\eta^2\biggl]Se^{-\frac{|X|^2}{2}}d\mu\\
 &\ \ \ +\frac{1+\alpha}{\frac{n+3}{n+1}-\epsilon}\int_M(-\lambda f_3\eta^2)e^{-\frac{|X|^2}{2}}d\mu\\
 &\leq \frac{1+\alpha}{\frac{n+3}{n+1}-\epsilon}\int_M \eta^2S^2e^{-\frac{|X|^2}{2}}d\mu
 +\frac{2n}{n+1}\cdot\frac{1+\alpha}{\frac{n+3}{n+1}-\epsilon}\int_M\eta^2|\nabla H|^2e^{-\frac{|X|^2}{2}}d\mu\\
 &\ \ \  +\int_M\biggl\{\bigl[\frac{1+\alpha}{\frac{n+3}{n+1}-\epsilon}\times \frac{1}{\epsilon}+1+\frac{1}{\alpha}\bigl]|\nabla \eta|^2+(1-\frac{1+\alpha}{\frac{n+3}{n+1}-\epsilon})\eta^2\\
&\ \ \ \ \ \ \ \ \ \ \ +\frac{\lambda}{H-\lambda}(1-\frac{1+\alpha}{\frac{n+3}{n+1}-\epsilon})\eta^2\biggl\}Se^{-\frac{|X|^2}{2}}d\mu.
\endaligned
\end{equation*}
Using
\begin{equation*}
\lambda\frac{S}{H-\lambda}\leq\lambda f_3\leq|\lambda|S\sqrt{S}\leq\frac{1}{2\delta}|\lambda| S^2+\frac{\delta}{2}|\lambda| S,
\end{equation*}
for $\delta>0$, we obtain, by taking $\alpha$ and  $\epsilon$ such that   $1-\frac{1+\alpha}{\frac{n+3}{n+1}-\epsilon}> 0$
\begin{equation}
\aligned
&\ \ \ \biggl(1-\frac{1+\alpha}{\frac{n+3}{n+1}-\epsilon}\biggl)\biggl(1-\frac{|\lambda|}{2\delta}\biggl)
\int_M\eta^2S^2e^{-\frac{|X|^2}{2}}d\mu\\
&\leq\frac{2n}{n+1}\frac{1+\alpha}{\frac{n+3}{n+1}-\epsilon}\int_M\eta^2|\nabla H|^2e^{-\frac{|X|^2}{2}}d\mu\\
&\ \ \ +\int_M\biggl\{\biggl(\frac{1+\alpha}{\frac{n+3}{n+1}-\epsilon}\frac{1}{\epsilon}+1+\frac{1}{\alpha}
 \biggl)|\nabla\eta|^2+\biggl(1-\frac{1+\alpha}{\frac{n+3}{n+1}-\epsilon}\biggl)\eta^2\\
&\ \ \ \ +\biggl(1-\frac{1+\alpha}{\frac{n+3}{n+1}-\epsilon}\biggl)\eta^2
\frac{\delta}{2}|\lambda|\biggl\}Se^{-\frac{|X|^2}{2}}d\mu.
\endaligned
\end{equation}
Assuming $|\eta|\leq1$ and $|\nabla \eta|\leq1$,
choosing  $\delta$ such that
$\frac{|\lambda|}{2\delta}< 1$,  we have
\begin{equation}\label{eq:224-4}
\int_M\eta^2S^2e^{-\frac{|X|^2}{2}}d\mu\leq C(n,\lambda)\int_M(|\nabla H|^2+S)e^{-\frac{|X|^2}{2}}d\mu
\end{equation}
for some constant $C(n,\lambda)$ depending on $n$ and $\lambda$. Since $|\nabla H|\leq\sqrt{S}|X|$ holds from  \eqref{eq:224-3},
one has from \eqref{eq:224-4}
\begin{equation}\label{eq:224-5}
\int_M\eta^2S^2e^{-\frac{|X|^2}{2}}d\mu\leq C(n,\lambda)\int_MS(1+|X|^2)e^{-\frac{|X|^2}{2}}d\mu.
\end{equation}
Since $H-\lambda> 0$ and $\lambda f_3\geq\lambda\frac{S}{H-\lambda}$, let $\eta_j$ be one on $B_j$ and cut off linearly to zero from $\partial B_j$ to $\partial B_{j+1}$, where
$B_j=X(M)\cap B_j(0)$ with $B_j(0)$ is the Euclidean ball of radius $j$ centered at the origin.
Applying the proposition \ref{proposition 10} with $\eta=\eta_j|X|$, letting $j\rightarrow \infty$,
the  dominated  convergence theorem and the polynomial  area growth give that $\int_MS(1+|X|^2)e^{-\frac{|X|^2}{2}}d\mu<+\infty$.
Thus \eqref{eq:224-5} and the
dominated convergence theorem give that
$$\int_MS^2e^{-\frac{|X|^2}{2}}d\mu< +\infty.
$$
Hence, from (\ref{eq:224-2}), we also have
$$\int_M|\nabla\sqrt{S}|^2e^{-\frac{|X|^2}{2}}d\mu< +\infty.$$
We next prove $\int_M\sum\limits_{i,j,k}h_{ijk}^2e^{-\frac{|X|^2}{2}}d\mu<+\infty$. From \eqref{eq:224-6}, one has
\begin{equation}
\aligned
& \ \ \ \int_M\eta^2\sum_{i,j,k}h_{ijk}^2e^{-\frac{|X|^2}{2}}d\mu\\
&=\int_M\eta^2(S^2-S)e^{-\frac{|X|^2}{2}}d\mu-\int_M\lambda f_3\eta^2e^{-\frac{|X|^2}{2}}d\mu\\
&\ \ \ -\int_M2\eta\sqrt{S}\langle \nabla\eta,\nabla\sqrt{S}\rangle e^{-\frac{|X|^2}{2}}d\mu\\
&\leq C_0(n,\lambda)\int_M(\eta^2S^2+\eta^2S+|\nabla\eta|^2|\nabla\sqrt{S}|^2)e^{-\frac{|X|^2}{2}}d\mu\\
&<+\infty,
\endaligned
\end{equation}
where $C_0(n,\lambda)$ is constant depending on $n$ and $\lambda$.
The dominated convergence theorem gives that
\begin{equation}
\int_M\sum\limits_{i,j,k}h_{ijk}^2e^{-\frac{|X|^2}{2}}d\mu< +\infty.
\end{equation}
This shows that

\begin{equation}\label{eq:324-1}
\int_M(S+S^2+|\nabla\sqrt{S}|^2+\sum_{i,j,k}h_{ijk}^2)e^{-\frac{|X|^2}{2}}d\mu< +\infty.
\end{equation}
From \eqref{eq:324-1}, we have
\begin{equation*}
\int_M(S^2+|\nabla \sqrt{S}|^2)e^{-\frac{|X|^2}{2}}d\mu<+\infty,
\end{equation*}
that is, $\sqrt{S}$ is in the weighted $W^{1,2}$ space. Applying the proposition \ref{proposition 10} with $\eta=\eta_j\sqrt{S}$, letting $j\rightarrow \infty$,
 using the  dominated  convergence theorem, one has
\begin{equation}
\int_MS|\nabla\log(H-\lambda)|^2e^{-\frac{|X|^2}{2}}d\mu<+\infty.
\end{equation}
It follows that
\begin{equation}
\aligned
&\int_M|\nabla S||\nabla\log(H-\lambda)|e^{-\frac{|X|^2}{2}}d\mu\\
&\leq\int_M(|\nabla\sqrt{S}|^2+S|\nabla\log(H-\lambda)|^2)
e^{-\frac{|X|^2}{2}}d\mu< +\infty.
\endaligned
\end{equation}
\eqref{eq:224-10} gives that
\begin{equation}
\aligned
&\ \ \ \int_MS|\mathcal{L}\log(H-\lambda)|e^{-\frac{|X|^2}{2}}d\mu\\
&=\int_MS\biggl|1-S+\frac{\lambda}{H-\lambda}-|\nabla\log(H-\lambda)|^2\biggl|e^{-\frac{|X|^2}{2}}d\mu\\
&\leq C_1(n,\lambda)\int_M\biggl\{S^2+S+S|\nabla\log(H-\lambda)|^2\biggl\}e^{-\frac{|X|^2}{2}}d\mu\\
&<+\infty,
\endaligned
\end{equation}
where $C_1(n,\lambda)$ is constant. Thus, we obtain
\begin{equation}
\int_M\biggl\{S|\nabla\log(H-\lambda)|+|\nabla S||\nabla\log(H-\lambda)|+S\mathcal{L}\log(H-\lambda)|\biggl\}
e^{-\frac{|X|^2}{2}}d\mu< +\infty.
\end{equation}
By  applying the corollary \ref{corollary 1} to $S$ and $\log(H-\lambda)$, we get
\begin{equation}\label{eq:225-2}
\aligned
&\ \ \ \int_M\langle\nabla S,\nabla\log(H-\lambda)\rangle e^{-\frac{|X|^2}{2}}d\mu\\
&=-\int_MS\mathcal{L}\log(H-\lambda)e^{-\frac{|X|^2}{2}}d\mu\\
&=\int_MS\biggl(S-1-\frac{\lambda}{H-\lambda}+|\nabla\log(H-\lambda)|^2\biggl)e^{-\frac{|X|^2}{2}}d\mu.
\endaligned
\end{equation}
On one hand,  \eqref{eq:224-9} gives
\begin{equation}
\aligned
&\int_M\sqrt{S}|\mathcal{L}\sqrt{S}|e^{-\frac{|X|^2}{2}}d\mu\\
&=\int_M\biggl|\sum_{i,j,k}h_{ijk}^2-|\nabla\sqrt{S}|^2
+S(1-S)+\lambda f_3\biggl|e^{-\frac{|X|^2}{2}}d\mu\\
&\leq C_2(n,\lambda)\int_M\biggl(\sum_{i,j,k}h_{ijk}^2+|\nabla\sqrt{S}|^2+S+S^2\biggl)e^{-\frac{|X|^2}{2}}d\mu\\
&<+\infty.
\endaligned
\end{equation}
Hence
\begin{equation}
\int_M\biggl(\sqrt{S}|\nabla\sqrt{S}|+|\nabla\sqrt{S}|^2+\sqrt{S}|\mathcal{L}\sqrt{S}|\biggl)
e^{-\frac{|X|^2}{2}}d\mu<+\infty.
\end{equation}
On the other hand, we have from \eqref{eq:224-9} and the lemma \ref{lemma 001}
\begin{equation}\label{eq:225-1}
\mathcal{L}\sqrt{S}\geq\sqrt{S}-\sqrt{S}S+\frac{\lambda f_3}{\sqrt{S}}.
\end{equation}
Then we can apply the corollary \ref{corollary 1}
to $\sqrt{S}$ and $\sqrt{S}$ and obtain
\begin{equation}\label{eq:225-3}
\aligned
&\ \ \ \int_M|\nabla\sqrt{S}|^2e^{-\frac{|X|^2}{2}}d\mu\\
&=-\int_M\sqrt{S}\mathcal{L}\sqrt{S}e^{-\frac{|X|^2}{2}}d\mu\\
&\leq\int_M(S^2-S-\lambda f_3)e^{-\frac{|X|^2}{2}}d\mu.
\endaligned
\end{equation}

\end{proof}

\vskip3mm
\noindent
{\it Proof of Theorem \ref{theorem 10}}. Since $H-\lambda\geq0$ and $\mathcal{L}H-H\leq0$, if $\lambda\leq0$, we have from the maximum principle that
either $H-\lambda\equiv0$ or $H-\lambda>0$, if $H-\lambda\equiv0$, \eqref{eq:224-3} and \eqref{eq:224-11} give
that $\lambda=0=H$, then $M$ is a  self-shrinker of the mean curvature flow.
According to the results of Colding and Minicozzi \cite{[CM]}, $M$ is $\mathbb{R}^{n}$.
If $\lambda>0$ and $H-\lambda=0$ at some point $p\in M$, then we
see from $\lambda(f_3(H-\lambda)-S)\geq0$ that $S=0$ and $H=0$ at $p$, then $\lambda\equiv0$,
according to the results of Colding and Minicozzi \cite{[CM]}, we know that $M$ is $\mathbb{R}^{n}$.
Hence, for any $\lambda$, we have either $M$ is $\mathbb{R}^{n}$ or $H-\lambda>0$.

\noindent Next, we assume that $H-\lambda>0$.
From the proposition \ref{proposition 11}, we have
\begin{equation}\label{eq:324-2}
\aligned
&\ \ \ \int_M\langle\nabla S,\nabla\log(H-\lambda)\rangle e^{-\frac{|X|^2}{2}}d\mu\\
&=-\int_MS\mathcal{L}\log(H-\lambda)e^{-\frac{|X|^2}{2}}d\mu\\
&=\int_MS\biggl(S-1-\frac{\lambda}{H-\lambda}+|\nabla\log(H-\lambda)|^2\biggl)e^{-\frac{|X|^2}{2}}d\mu,
\endaligned
\end{equation}
and
\begin{equation}\label{eq:324-3}
\aligned
&\ \ \ \int_M|\nabla\sqrt{S}|^2e^{-\frac{|X|^2}{2}}d\mu\\
&=-\int_M\sqrt{S}\mathcal{L}\sqrt{S}e^{-\frac{|X|^2}{2}}d\mu\\
&\leq\int_M(S^2-S-\lambda f_3)e^{-\frac{|X|^2}{2}}d\mu.
\endaligned
\end{equation}
Substituting \eqref{eq:324-3} into \eqref{eq:324-2} and using $\lambda f_3\geq\lambda\frac{S}{H-\lambda}$, one has
\begin{equation}
\aligned
0&\geq\int_M\biggl\{|\nabla\sqrt{S}|^2-2\sqrt{S}\langle \nabla\sqrt{S},\nabla\log(H-\lambda)\rangle +S|\nabla\log(H-\lambda)|^2\\
&\ \ \ \ \ +\lambda f_3-\lambda\frac{S}{H-\lambda}\biggl\}e^{-\frac{|X|^2}{2}}d\mu\\
&\geq\int_M\bigl|\nabla\sqrt{S}-\sqrt{S}\nabla\log(H-\lambda)\bigl|^2e^{-\frac{|X|^2}{2}}d\mu.
\endaligned
\end{equation}
Hence we conclude that $\nabla\sqrt{S}=\sqrt{S}\nabla\log(H-\lambda)$. Thus, we obtain
\begin{equation}
\sqrt{S}=\beta(H-\lambda)
\end{equation}
for a constant $\beta>0$. Since all inequalities in above equations become equalities, we obtain
\begin{equation}\label{eq:225-4}
\sum_{i,j,k}h_{ijk}^2=|\nabla\sqrt{S}|^2,\ \ \lambda f_3=\lambda\frac{S}{H-\lambda}.
\end{equation}
From the lemma \ref{lemma 001} and \eqref{eq:225-4}, we know

\vskip2mm
(1) There is a constant $C_k$ such that $h_{iik}=C_k\lambda_i$ for every $i$ and $k$.

\vskip2mm
(2) If $i\neq j$, then $h_{ijk}=0$, that is, $h_{ijk}=0$ unless $i=j=k$ since $h_{ijk}=h_{ikj}$.

\vskip2mm
\noindent
If $\lambda_i\neq0$ and $j\neq i$, then $0=h_{iij}=C_j\lambda_i$.  It follows that $C_j=0$. If the rank of matrix
$(h_{ij})$ is at least two at $p$, then $C_j=0$ for $j\in\{1,2,\cdots,n\}$. Hence,  we have $h_{ijk}(p)=0$.

\noindent
We next consider two cases.

\vskip4mm
\noindent {\bf Case 1}: The rank of matrix $(h_{ij})$ is greater than one at $p$.
\vskip2mm

\noindent
In this case, we will prove that the rank of $(h_{ij})$ is at least two everywhere. In fact, for $q\in M$, let
$\lambda_1(q)$ and $\lambda_2(q)$ be the two eigenvalues of $(h_{ij})(q)$ that are largest in absolute value and define the set
\begin{equation}
\Omega=\{q\in M|\lambda_1(q)=\lambda_1(p), \lambda_2(q)=\lambda_2(p)\}.
\end{equation}
Then $p\in\Omega$, since $\lambda_i$'s are continuous, so $\Omega$ is closed.
Given any point $q\in \Omega$, it follows that the rank of $(h_{ij})$ is at least two at $q$.
Hence there is an open set $U$, $q\in U $, where the rank of $(h_{ij})$ is at least two.
On $U$, we have $h_{ijk}=0$ and the eigenvalues of $(h_{ij})$
are constant on $U$. Thus, $U\subset\Omega$, $\Omega$ is open. Since $M$ is connected, we have $\Omega=M$
and $h_{ijk}\equiv0$ on $M$. We know that $M=S^k(r)\times \mathbb{R}^{n-k}$, where $k>1$.

\vskip4mm
\noindent {\bf Case 2}: The rank of matrix $(h_{ij})$ is one.

\vskip2mm
\noindent
From Case 1, we know that the rank of $(h_{ij})$ is one everywhere. Hence $S=H^2$.
On the other hand, $S=\beta^2(H-\lambda)^2$, hence $H^2=\beta^2(H-\lambda)^2$.
If $\lambda=0$, then $M$ is a  self-shrinker of the mean curvatue flow. If $\lambda\neq0$, then we have $H$ is constant.
$M$ is $S^1(r)\times \mathbb{R}^{n-1}$ from the proposition \ref{prop 2.2}. This completes the proof of Theorem \ref{theorem 10}.
$$\eqno{\Box}$$

\vskip 5mm

\section{Properness and polynomial  area growth for $\lambda$-hypersurfaces}

\noindent
For $n$-dimensional complete and non-compact Riemannian manifolds with nonnegative Ricci curvature, the well-known
theorem of Bishop and Gromov says that  geodesic balls have at most polynomial area growth:
$$
{\rm Area} (B_r(x_0))\leq Cr^n.
$$
For $n$-dimensional complete and non-compact gradient shrinking Ricci soliton, Cao and Zhou \cite{CZ} have proved
geodesic balls have at most polynomial area growth.
For self-shrinkers, Ding and Xin \cite{[DX1]} proved that any complete non-compact properly immersed self-shrinker in the Euclidean space has polynomial area growth. X. Cheng and Zhou \cite{[CZ]} showed that any complete immersed self-shrinker with polynomial area growth in the Euclidean space is proper. Hence any complete immersed self-shrinker is proper if and only if it has polynomial area growth.

\noindent It is our purposes in this section to study the area growth for
$\lambda$-hypersurfaces.
First of all,  we study the equivalence of properness and polynomial area growth for  $\lambda$-hypersurfaces.
If $X: M\rightarrow \mathbb{R}^{n+1}$ is an $n$-dimensional hypersurface in $\mathbb{R}^{n+1}$,
we say $M$ has polynomial area growth if there exist constant $C$ and $d$ such that for all $r\geq 1$,
\begin{equation}
{\rm Area}(B_r(0)\cap X(M))=\int_{B_r(0)\cap X(M)}d\mu\leq Cr^d,
\end{equation}
where $B_r(0)$ is a round ball in $\mathbb{R}^{n+1}$ with radius $r$ and centered at the origin.

\begin{theorem}\label{theorem 9.1}
Let $X: M\rightarrow \mathbb{R}^{n+1}$ be a complete and non-compact properly immersed
$\lambda$-hypersurface in the Euclidean space $\mathbb{R}^{n+1}$.
Then, there is a positive constant $C$ such that for $r\geq1$,
\begin{equation}
{\rm Area}(B_r(0)\cap X(M))=\int_{B_r(0)\cap X(M)}d\mu\leq Cr^{n+\frac{\lambda^2}2-2\beta-\frac{\inf H^2}2},
\end{equation}
where  $\beta=\frac{1}{4}\inf(\lambda-H)^2$.
\end{theorem}
\begin{proof}
Since $X: M\rightarrow \mathbb{R}^{n+1}$
is  a complete and non-compact properly immersed
$\lambda$-hypersurface in the Euclidean space $\mathbb{R}^{n+1}$, we
have
$$
\langle X, N\rangle +H=\lambda.
$$
Defining $f=\frac{|{X}|^2}{4}$, we have
\begin{equation}
f-|{\nabla}f|^2=\frac{|{X}|^2}{4}-\frac{|{X}^{T}|^2}{4}
=\frac{|{X}^{\perp}|^2}{4}=\frac14(\lambda-{H})^2,
\end{equation}
\begin{equation}
\aligned
{\Delta} f&=\frac{1}{2}(n+{H}\langle N,{X}\rangle) \\
        &=\frac{1}{2}(n+\lambda \langle N,{X}\rangle-\langle N,{X}\rangle^2)\\
        &=\frac{1}{2}n+\frac{\lambda^2}4-\frac{H^2}4-f+|{\nabla}f|^2.
\endaligned
\end{equation}
Hence, we obtain
\begin{equation}
|{\nabla}{(f-\beta)}|^2\leq (f-\beta),
\end{equation}
\begin{equation}
{\Delta}{(f-\beta)}-|{\nabla}{(f-\beta)}|^2+ {(f-\beta)}\leq (\frac{n}{2}+\frac{\lambda^2}4-\beta -\frac{\inf H^2}4 ).
\end{equation}
Since the  immersion $X$ is proper, we know that $\overline{f}=f-\beta$ is proper.
Applying the theorem 2.1 of X. Cheng and Zhou \cite{[CZ]} to $\overline{f}=f-\beta$
with $k=(\frac{n}{2}+\frac{\lambda^2}4-\beta-\frac{\inf H^2}4)$, we obtain
 \begin{equation}
{\rm Area}(B_r(0)\cap X(M))=\int_{B_r(0)\cap X(M)}d\mu\leq Cr^{n+\frac{\lambda^2}2-2\beta-\frac{\inf H^2}2},
\end{equation}
where  $\beta=\frac{1}{4}\inf(\lambda-H)^2$ and $C$ is  a constant.
\end{proof}

\begin{remark}
The estimate in our theorem \ref{theorem 9.1} is the best possible because the cylinders $S^k(r_0)\times \mathbb{R}^{n-k}$
satisfy the equality.
\end{remark}
\begin{remark}
By making use of the same assertions as  in X. Cheng and Zhou \cite{[CZ]} for self-shrinkers,  we can prove
the weighted area of  a complete and non-compact properly immersed
$\lambda$-hypersurface in the Euclidean space $\mathbb{R}^{n+1}$
is bounded.
\end{remark}

\noindent
By making use of  to the same assertions  as in X. Cheng and Zhou \cite{[CZ]} for self-shrinkers,
we can prove the following theorem. We will leave it  for readers.
\begin{theorem}
If  $X: M\rightarrow \mathbb{R}^{n+1}$ is  an $n$-dimensional complete immersed $\lambda$-hypersurface with polynomial area  growth, then
$X: M\rightarrow \mathbb{R}^{n+1}$ is proper.
\end{theorem}

\vskip 2mm

\section{A lower bound growth of the area for  $\lambda$-hypersurfaces}

\noindent
For $n$-dimensional complete and non-compact Riemannian manifolds with nonnegative Ricci curvature, the well-known
theorem of Calabi and Yau says that geodesic balls have at least linear area growth:
$$
{\rm Area} (B_r(x_0))\geq Cr.
$$
Cao and Zhu \cite{CZu} have proved that  $n$-dimensional complete and non-compact gradient shrinking Ricci soliton
must have infinite volume. Furthermore, Munteanu and Wang \cite{MW} have proved that
areas of geodesic balls for  $n$-dimensional complete and non-compact gradient shrinking Ricci soliton
 has at least linear growth. For self-shrinkers, Li and Wei \cite{[LW2]} proved that any complete and non-compact proper self-shrinker has at least linear area growth.

\noindent In this section, we study the lower bound growth of the area for $\lambda$-hypersurfaces.
The following lemmas play a very important role in order to prove our results.

\begin{lemma}\label{lemma:4-20-1}
Let $X: M\rightarrow \mathbb{R}^{n+1}$ be an $n$-dimensional complete noncompact proper $\lambda$-hypersurface, then
there exist constants $C_1(n,\lambda)$ and $c(n,\lambda)$ such that for all $t\geq C_1(n,\lambda)$,
\begin{equation}
{\rm Area}(B_{t+1}(0)\cap X(M))-{\rm Area}(B_t(0)\cap X(M))\leq c(n,\lambda)\frac{{\rm Area}(B_t(0)\cap X(M))}{t}
\end{equation}
and
\begin{equation}
{\rm Area}(B_{t+1}(0)\cap X(M))\leq 2{\rm Area}(B_{t}(0)\cap X(M)).
\end{equation}
\end{lemma}
\begin{proof}
Since $X: M\rightarrow \mathbb{R}^{n+1}$ is a complete $\lambda$-hypersurface, one has
\begin{equation}\label{eq:4-20-1}
\frac{1}{2}\Delta |X|^2=n+H\langle N,X\rangle=n+H\lambda-H^2.
\end{equation}
Integrating \eqref{eq:4-20-1} over $B_r(0)\cap X(M)$, we obtain
\begin{equation}\label{eq:4-20-5}
\aligned
&\ \ \ \ n{\rm Area}(B_{r}(0)\cap X(M))+\int_{B_r(0)\cap X(M)}H\lambda d\mu-\int_{B_r(0)\cap X(M)}H^2d\mu\\
 &=\frac{1}{2}\int_{B_r(0)\cap X(M)}\triangle |X|^2d\mu\\
&=\frac{1}{2}\int_{\partial(B_r(0)\cap X(M))}\nabla |X|^2\cdot\frac{\nabla \rho}{|\nabla\rho|}d\sigma\\
&=\int_{\partial(B_r(0)\cap X(M))}|X^{T}|d\sigma\\
&=\int_{\partial(B_r(0)\cap X(M))}\frac{|X|^2-(\lambda-H)^2}{|X^{T}|}d\sigma\\
&=r({\rm Area}(B_{r}(0)\cap X(M)))^{'}-\int_{\partial(B_r(0)\cap X(M))}\frac{(\lambda-H)^2}{|X^{T}|}d\sigma,
\endaligned
\end{equation}
where $\rho(x):=|X(x)|$, $\nabla \rho=\frac{X^T}{|X|}$.
Here we used, from the co-area formula,
\begin{equation}\label{eq:4-20-5-1}
\bigl({\rm Area}(B_{r}(0)\cap X(M))\bigl)^{'}=r\int_{\partial(B_r(0)\cap X(M))}\frac{1}{|X^{T}|}d\sigma.
\end{equation}
Hence, we obtain
\begin{equation}\label{eq:4-20-2}
\aligned
&\ \ \ \ (n+\frac{\lambda^2}{4}){\rm Area}(B_{r}(0)\cap X(M))-r({\rm Area}(B_{r}(0)\cap X(M)))^{'}\\
&=\int_{B_r(0)\cap X(M)}(H-\frac{\lambda}{2})^2d\mu-\int_{\partial(B_r(0)\cap X(M))}\frac{(\lambda-H)^2}{|X^{T}|}d\sigma,
\endaligned
\end{equation}
From \eqref{eq:4-20-5-1},   $(H-\lambda)^2=\langle N,{X}\rangle^2\leq |X|^2=r^2$ on $\partial(B_r(0)\cap X(M))$
and  \eqref{eq:4-20-2},
we conclude
\begin{equation}
\int_{B_{r}(0)\cap X(M)}(H-\frac{\lambda}{2})^2d\mu\leq (n+\frac{\lambda^2}{4}){\rm Area}(B_{r}(0)\cap X(M)).
\end{equation}
Furthermore, we have
\begin{equation}\label{eq:4-20-6}
\aligned
\int_{B_{r}(0)\cap X(M)}(H-\lambda)^2d\mu & \leq \int_{B_{r}(0)\cap X(M)}2\bigl[(H-\frac{\lambda}{2})^2
  +\frac{\lambda^2}{4}\bigl]d\mu\\
&\leq (2n+\lambda^2){\rm Area}(B_{r}(0)\cap X(M)).
\endaligned
\end{equation}
\eqref{eq:4-20-2} implies that
\begin{equation}\label{eq:4-20-3}
\aligned
&\ \ \ \ \bigl(r^{-n-\frac{\lambda^2}{4}}{\rm Area}(B_{r}(0)\cap X(M))\bigl)^{'}\\
&=r^{-n-1-\frac{\lambda^2}{4}}\biggl(r\bigl({\rm Area}(B_{r}(0)\cap X(M))\bigl)^{'}-(n+\frac{\lambda^2}{4})
{\rm Area}(B_{r}(0)\cap X(M))\biggl)\\
&=r^{-n-1-\frac{\lambda^2}{4}}\int_{\partial(B_r(0)\cap X(M))}\frac{(H-\lambda)^2}{|X^{T}|}d\sigma
 -r^{-n-1-\frac{\lambda^2}{4}}\int_{B_r(0)\cap X(M)}(H-\frac{\lambda}{2})^2d\mu.
\endaligned
\end{equation}
Integrating \eqref{eq:4-20-3} from $r_2$ to $r_1$ ($r_1>r_2$), one has
\begin{equation}\label{eq:4-20-4}
\aligned
&\ \ \ \ r_1^{-n-\frac{\lambda^2}{4}}{\rm Area}(B_{r_1}(0)\cap X(M))-
         r_2^{-n-\frac{\lambda^2}{4}}{\rm Area}(B_{r_2}(0)\cap X(M))\\
&=r_1^{-n-2-\frac{\lambda^2}{4}}\int_{B_{r_1}(0)\cap X(M)}(H-\lambda)^2d\mu
  -r_2^{-n-2-\frac{\lambda^2}{4}}\int_{B_{r_2}(0)\cap X(M)}(H-\lambda)^2d\mu\\
&\ \ \ \
  +(n+2+\frac{\lambda^2}{4})\int_{r_2}^{r_1}s^{-n-3-\frac{\lambda^2}{4}}(\int_{B_{s}(0)\cap X(M)}(H-\lambda)^2d\mu)ds\\
&\ \ \ \ -\int_{r_2}^{r_1}s^{-n-1-\frac{\lambda^2}{4}}(\int_{B_{s}(0)\cap X(M)}(H-\frac{\lambda}{2})^2d\mu)ds\\
&\leq \bigl(r_1^{-n-2-\frac{\lambda^2}{4}}
 +r_2^{-n-2-\frac{\lambda^2}{4}}\bigl)\int_{B_{r_1}(0)\cap X(M)}(H-\lambda)^2d\mu.
\endaligned
\end{equation}
Here we used
$$
\biggl(\int_{B_{r}(0)\cap X(M)}(H-\lambda)^2d\mu\biggl)^{'} =r\int_{\partial(B_r(0)\cap X(M))}\frac{(H-\lambda)^2}{|X^{T}|}d\sigma
$$
and ${\rm Area}(B_{r}(0)\cap X(M))$ is non-decreasing in $r$ from  \eqref{eq:4-20-5-1}.
Combining \eqref{eq:4-20-4} with  \eqref{eq:4-20-6}, we have
\begin{equation}\label{eq:4-20-7}
\aligned
&\ \ \ \
\frac{{\rm Area}(B_{r_1}(0)\cap X(M))}{r_1^{n+\frac{\lambda^2}{4}}}
-\frac{{\rm Area}(B_{r_2}(0)\cap X(M))}{r_2^{n+\frac{\lambda^2}{4}}}\\
&\leq (2n+\lambda^2)\bigl(\frac{1}{r_1^{n+2+\frac{\lambda^2}{4}}}
+ \frac{1}{r_2^{n+2+\frac{\lambda^2}{4}}}\bigl){\rm Area}(B_{r_1}(0)\cap X(M)).
\endaligned
\end{equation}
Putting $r_1=t+1$, $r_2=t>0$, we get
\begin{equation}\label{eq:4-20-8}
\aligned
 &\biggl(1-\frac{2(2n+\lambda^2)(t+1)^{n+\frac{\lambda^2}{4}}}
                                     {t^{n+2+\frac{\lambda^2}{4}}}\biggl){\rm Area}(B_{{t+1}}(0)\cap X(M)) \\
                                     & \leq {\rm Area}(B_{{t}}(0)\cap X(M)) (\frac{t+1}{t})^{n+\frac{\lambda^2}{4}}.
                                     \endaligned
\end{equation}
For $t$ sufficiently large, one has, from \eqref{eq:4-20-8},
\begin{equation}\label{eq:4-20-9}
\aligned
&\ \ \ \
{\rm Area}(B_{t+1}(0)\cap X(M))-{\rm Area}(B_{t}(0)\cap X(M))\\
&\leq {\rm Area}(B_{t}(0)\cap X(M))\biggl ((1+\frac{1}{t})^n-1+\frac{C(t+1)^{2n+\lambda^2}{4}}{t^{2n+2+\lambda^2}}\biggl),
\endaligned
\end{equation}
where ${C}$ is constant only depended on $n$, $\lambda$.
Therefore,  there exists some constant $C_1(n,\lambda)$ such that for all $t\geq C_1(n,\lambda)$,
\begin{equation}
\aligned
&{\rm Area}(B_{t+1}(0)\cap X(M))-{\rm Area}(B_{t}(0)\cap X(M))\\
&\leq c(n,\lambda)\frac{{\rm Area}(B_{t}(0)\cap X(M))}{t},
\endaligned
\end{equation}
\begin{equation}
{\rm Area}(B_{t+1}(0)\cap X(M))\leq 2{\rm Area}(B_{t}(0)\cap X(M)),
\end{equation}
where $c(n,\lambda)$ depends only on $n$ and $\lambda$. This completes the proof of the lemma \ref{lemma:4-20-1}.

\end{proof}

\vskip3mm
\noindent
Using Logarithmic Sobolev inequality for hypersurfaces in Euclidean space due to Ecker \cite{[E]} and conformal theory, we can show
\begin{lemma}\label{lemma:4-20-2}
Let $X: M\rightarrow \mathbb{R}^{n+1}$ be an $n$-dimensional hypersurface with measure $d\mu$. Then the following inequality
\begin{equation}\label{eq:4-20-10}
\aligned
&\ \ \ \
\int_M f^2(\ln f^2)e^{-\frac{|X|^2}{2}}d\mu-\int_M f^2e^{-\frac{|X|^2}{2}}d\mu \ln (\int_M f^2e^{-\frac{|X|^2}{2}}2^{\frac{n}{2}}d\mu)\\
&\leq \int_M|\nabla f|^2e^{-\frac{|X|^2}{2}}d\mu+\frac{1}{4}\int_M|H+\langle X,N\rangle|^2f^2e^{-\frac{|X|^2}{2}}d\mu\\
&\ \ \  +C(n)\int_Mf^2e^{-\frac{|X|^2}{2}}d\mu
\endaligned
\end{equation}
holds for any nonnegative function $f$ for which all integrals are well-defined and finite, where $C(n)$ is a positive constant depending on $n$.
\end{lemma}

\begin{corollary}\label{corollary:4-27-1}
For  an $n$-dimensional $\lambda$-hypersurface $X: M\rightarrow \mathbb{R}^{n+1}$,  we have  the following inequality
\begin{equation}
\int_M f^2(\ln f)e^{-\frac{|X|^2}{2}}d\mu\leq \frac{1}{2}\int_M|\nabla f|^2e^{-\frac{|X|^2}{2}}d\mu
+(\frac{1}{2}C(n)+\frac{1}{8}\lambda^2)2^{-\frac{n}{2}}
\end{equation}
 for any nonnegative function $f$ which satisfies
\begin{equation}
\int_M f^2e^{-\frac{|X|^2}{2}}2^{\frac{n}{2}}d\mu=1.
\end{equation}
\end{corollary}

\begin{corollary}\label{corollary:4-27-2}
If  $X: M\rightarrow \mathbb{R}^{n+1}$ is an $n$-dimensional $\lambda$-hypersurface,  then the following inequality
\begin{equation}
\aligned
&\ \ \ \
\int_M u^2(\ln u^2)d\mu-\int_M u^2d\mu \ln (\int_M u^2d\mu)\\
&\leq 2\int_M|\nabla u|^2d\mu+(\frac{1}{4}\lambda^2+\frac{n}{2}\ln 2+C(n))\int_Mu^2d\mu
\endaligned
\end{equation}
holds for any nonnegative function $f$ which satisfies
\begin{equation}
f=ue^{\frac{|X|^2}{4}}.
\end{equation}
\end{corollary}


\begin{lemma}$($\cite{[LW2]}$)$\label{lemma:4-27-1}
Let $X: M\rightarrow \mathbb{R}^{n+1}$ be a complete properly immersed hypersurface.
For any $x_0\in M$, $r\leq 1$, if $|H|\leq\frac{C}{r}$ in $B_r(X(x_0))\cap X(M)$ for some constant $C>0$. Then
\begin{equation}
\text{Area}(B_{r}(X(x_0))\cap X(M))\geq \kappa r^n,
\end{equation}
where $\kappa=\omega_n e^{-C}$.
\end{lemma}

\begin{lemma}\label{lemma:4-21-1}
If $X: M\rightarrow \mathbb{R}^{n+1}$ is an $n$-dimensional complete and non-compact proper $\lambda$-hypersurface. then it has infinite area.
\end{lemma}

\begin{proof}
Let
\begin{equation*}
\Omega(k_1,k_2)=\{x\in M: 2^{k_1-\frac{1}{2}}\leq \rho(x)\leq 2^{k_2-\frac{1}{2}}\},
\end{equation*}
\begin{equation*}
A(k_1,k_2)=\text{Area}(X(\Omega(k_1,k_2))),
\end{equation*}
where $\rho(x)=|X(x)|$. Since $X: M\rightarrow \mathbb{R}^{n+1}$ is a complete and non-compact proper immersion, $X(M)$ can not
be contained in a compact Euclidean ball. Then, for $k$ large enough, $\Omega(k, k+1)$ contains at least $2^{2k-1}$ disjoint balls
\begin{equation*}
B_r(x_i)=\{ x\in M: \rho_{x_i}(x)<2^{-\frac{1}{2}}r\}, \ \ x_i\in M,\ r=2^{-k}
\end{equation*}
where $\rho_{x_i}(x)=|X(x)-X(x_i)|$. Since, in $\Omega(k,k+1)$,
\begin{equation}
|H|\leq |H-\lambda|+|\lambda|=|\langle X,N\rangle|+|\lambda|\leq |X|+|\lambda|\leq 2^{k}\sqrt2+|\lambda|
\leq \frac{\sqrt2+|\lambda|}{r},\end{equation}
by using of the lemma \ref{lemma:4-27-1}, we get
\begin{equation}\label{eq:4-27-2}
A(k,k+1)\geq \kappa_1 2^{2k-1-kn},
\end{equation}
with $\kappa_1=\omega_n e^{-(\sqrt{2}+|\lambda|)2^{-\frac{1}{2}}}2^{-\frac{n}{2}}$.

{\bf Claim}: If ${\text Area}(X(M))<\infty$, then, for every $\varepsilon>0$, there exists a large constant $k_0>0$ such that,
\begin{equation}\label{eq:4-27-1}
A(k_1,k_2)\leq \varepsilon \ \ \ {\rm and}\ \ \ A(k_1,k_2)\leq 2^{4n}A(k_1+2,k_2-2), \ \ \ {\rm if }\ \ k_2>k_1>k_0.
\end{equation}

\noindent
In fact, we may choose $K>0$ sufficiently large such that  $k_1\approx \frac{K}{2}$, $k_2\approx\frac{3K}{2}$.
Assume  \eqref{eq:4-27-1} does not hold, that is,
\begin{equation*}
A(k_1,k_2)\geq 2^{4n}A(k_1+2,k_2-2).
\end{equation*}
If
\begin{equation*}
A(k_1+2,k_2-2)\leq 2^{4n}A(k_1+4,k_2-4),
\end{equation*}
then we complete the proof of the claim. Otherwise, we can repeat the procedure for  $j$ times, we have
\begin{equation*}
A(k_1,k_2)\geq 2^{4nj}A(k_1+2j,k_2-2j).
\end{equation*}
When $j\approx\frac{K}{4}$, we have from \eqref{eq:4-27-2}
\begin{equation*}
\text{Area}(X(M))\geq A(k_1,k_2)\geq 2^{nK}A(K,K+1)\geq \kappa_1 2^{2K-1}.
\end{equation*}
Thus, \eqref{eq:4-27-1}  must hold for some $k_2>k_1$
because $\text{Area}(M)<\infty$.
Hence for any $\varepsilon>0$, we can choose $k_1$ and $k_2\approx 3k_1$ such that \eqref{eq:4-27-1} holds.

\noindent
We  define a smooth cut-off function $\psi(t)$ by
\begin{equation}
 \psi(t)= {\begin{cases}
      \ 1, & \ \ 2^{k_1+\frac{3}{2}}\leq t\leq 2^{k_2-\frac{5}{2}},\\
      \ 0,& \ \ {\rm outside}\ [2^{k_1-\frac{1}{2}}, 2^{k_2-\frac{1}{2}}].\\
     \end{cases}}
     \ \ \ \ \ 0\leq \psi(t)\leq 1, \ \ \ \ |\psi^{'}(t)|\leq 1.
  \end{equation}
Letting
\begin{equation}
f(x)=e^{L+\frac{|X|^2}{4}}\psi(\rho(x)),
\end{equation}
we choose $L$ satisfying
\begin{equation}
1=\int_M f^2e^{-\frac{|X|^2}{2}}2^{\frac{n}{2}}d\mu=e^{2L}\int_{\Omega(k_1,k_2)}\psi^2(\rho(x))2^{\frac{n}{2}}d\mu.
\end{equation}
We obtain from the corollary \ref{corollary:4-27-1} and $t\ln t\geq-\frac{1}{e}$ for $0\leq t\leq 1$
\begin{equation}
\aligned
(\frac{1}{2}C(n)+\frac{1}{8}\lambda^2)2^{-\frac{n}{2}}&\geq\int_{\Omega(k_1,k_2)}e^{2L}\psi^2(L+\frac{|X|^2}{4}+\ln \psi)d\mu\\
              &\ \ \ -\frac{1}{2}\int_{\Omega(k_1,k_2)}e^{2L}|\psi^{'}\nabla \rho+\psi\frac{X^T}{2}|^2d\mu\\
              &\geq\int_{\Omega(k_1,k_2)}e^{2L}\psi^2(L+\frac{|X|^2}{4}+\ln \psi)d\mu\\
              &\ \ -\int_{\Omega(k_1,k_2)}e^{2L}|\psi^{'}|^2d\mu-\frac{1}{4}\int_{\Omega(k_1,k_2)}e^{2L}\psi^2|X|^2d\mu\\
              &=2^{-\frac{n}{2}}L+\int_{\Omega(k_1,k_2)}e^{2L}\psi^2\ln \psi d\mu-\int_{\Omega(k_1,k_2)}e^{2L}|\psi^{'}|^2 d\mu\\
              &\geq 2^{-\frac{n}{2}} L-(\frac{1}{2e}+1)e^{2L}A(k_1,k_2).
\endaligned
\end{equation}
Therefore,  it follows from \eqref{eq:4-27-1} that
\begin{equation}\label{eq:4-27-3}
\aligned
(\frac{1}{2}C(n)+\frac{1}{8}\lambda^2)2^{-\frac{n}{2}}&\geq 2^{-\frac{n}{2}} L-(\frac{1}{2e}+1)e^{2L}2^{4n}A(k_1+2,k_2-2)\\
              & \geq 2^{-\frac{n}{2}} L-(\frac{1}{2e}+1)e^{2L}2^{4n}\int_{\Omega(k_1,k_2)}\psi^2(\rho(x))d\mu\\
              &=2^{-\frac{n}{2}} L-(\frac{1}{2e}+1)2^{4n}2^{-\frac{n}{2}}.
\endaligned
\end{equation}
On the other hand, we have, from \eqref{eq:4-27-1} and definition of $f(x)$,
\begin{equation}
1\leq e^{2L}\varepsilon 2^{\frac{n}{2}}.
\end{equation}
Letting $\varepsilon>0$ sufficiently small, then $L$ can be arbitrary large, which  contradicts
 \eqref{eq:4-27-3}. Hence,  $M$ has infinite area.
\end{proof}

\begin{theorem}\label{theorem:4-21-1}
Let $X: M\rightarrow \mathbb{R}^{n+1}$ be an $n$-dimensional complete proper $\lambda$-hypersurface.
Then, for any $p\in M$, there exists a constant $C>0$ such that
$$
\text{Area}(B_{r}(X(x_0))\cap X(M))\geq Cr,
$$
for all $r>1$.
\end{theorem}

\begin{proof}
We can choose $r_0>0$ such that ${\rm Area}(B_{r}(0)\cap X(M))>0$ for $r\geq r_0$. It is sufficient to prove there exists a constant $C>0$ such that
\begin{equation}\label{eq:4-21-1}
{\rm Area}(B_{r}(0)\cap X(M))\geq Cr
\end{equation}
holds for all $r\geq r_0$. In fact, if \eqref{eq:4-21-1} holds, then for any $x_0\in M$ and $r>|X(x_0)|$,
\begin{equation}
B_r(X(x_0))\supset B_{r-|X(x_0)|}(0),
\end{equation}
and
\begin{equation}\label{eq:4-21-2}
{\rm Area}(B_{r}(X(x_0))\cap X(M)) \geq {\rm Area}(B_{r-|X(x_0)|}(0)\cap X(M))\geq \frac{C}{2}r,
\end{equation}
for $r\geq 2|X(x_0)|$.

\noindent
We next prove \eqref{eq:4-21-1} by contradiction. Assume for any $\varepsilon>0$, there exists $r\geq r_0$ such that
\begin{equation}\label{eq:4-21-3}
{\rm Area}(B_{r}(0)\cap X(M))\leq \varepsilon r.
\end{equation}
Without loss of generality, we assume $r\in \mathbb{N}$ and consider a set:
\begin{equation*}
D:=\{k\in \mathbb{N}: {\rm Area}(B_{t}(0)\cap X(M))\leq 2\varepsilon t {\rm \ for\ any \  integer}\  t  \ {\rm satisfying } \ r\leq t\leq k\}.
\end{equation*}
Next, we will show that  $k\in D$  for any integer $k$ satisfying $k\geq r$.
For $t\geq r_0$, we define a function $u$ by

\begin{equation}
 u(x)= {\begin{cases}
      \ t+2-\rho(x), & \ \ {\rm in}  \ \  B_{t+2}(0)\cap X(M)\setminus B_{t+1}(0)\cap X(M),\\
      \ 1,& \ \ {\rm in} \ \ B_{t+1}(0)\cap X(M)\setminus B_{t}(0)\cap X(M),\\
      \ \rho(x)-(t-1), & \ \ {\rm in}  \ \  B_{t}(0)\cap X(M)\setminus B_{t-1}(0)\cap X(M),\\
      \ 0, & \ \ {\rm otherwise}.\\
     \end{cases}}
  \end{equation}
Using the corollary \ref{corollary:4-27-2}, $|\nabla \rho|\leq 1$ and $t\ln t\geq -\frac{1}{e}$ for $0\leq t\leq 1$, we have
\begin{equation}\label{eq:4-21-5}
\aligned
&\ \
-\bigl(\int_M u^2d\mu\bigl)\ln \bigl\{\bigl({\rm Area}(B_{t+2}(0)\cap X(M))-{\rm Area}(B_{t-1}(0)\cap X(M))\bigl)2^{\frac{n}{2}}\bigl\}\\
&\leq C_0 \biggl({\rm Area}(B_{t+2}(0)\cap X(M))-{\rm Area}(B_{t-1}(0)\cap X(M))\biggl),
\endaligned
\end{equation}
where $C_0=2+\frac{1}{e}+\frac{\lambda^2}{4}+\frac{n}{2}\ln 2+C(n)$, $C(n)$ is the constant of the corollary \ref{corollary:4-27-2}.

\noindent For all $t\geq C_1(n,\lambda)+1$, we have from the lemma \ref{lemma:4-20-1}
\begin{equation}\label{eq:4-21-4}
\aligned
&\ \ {\rm Area}(B_{t+2}(0)\cap X(M))-{\rm Area}(B_{t-1}(0)\cap X(M))\\
&\leq c(n,\lambda)\biggl(\frac{{\rm Area}(B_{t+1}(0)\cap X(M))}{t+1}\\
&\ \ \ +\frac{{\rm Area}(B_{t}(0)\cap X(M))}{t}+\frac{{\rm Area}(B_{t-1}(0)\cap X(M))}{t-1}\biggl)\\
&\leq c(n,\lambda)\biggl(\frac{2}{t+1}+\frac{1}{t}+\frac{1}{t}(1+\frac{1}{C_1(n,\lambda)})\biggl){\rm Area}(B_{t}(0)\cap X(M))\\
&\leq C_2(n,\lambda)\frac{{\rm Area}(B_{t}(0)\cap X(M))}{t},
\endaligned
\end{equation}
where $C_2(n,\lambda)$ is constant depended only on $n$ and $\lambda$. Note that we can assume $r\geq C_1(n,\lambda)+1$
for the $r$ satisfying \eqref{eq:4-21-3}. In fact, if for any given $\varepsilon>0$, all the $r$ which satisfies \eqref{eq:4-21-3}
is bounded above by $C_1(n,\lambda)+1$, then ${\rm Area}(B_{r}(0)\cap X(M))\geq C r$ holds for
any $r>C_1(n,\lambda)+1$. Thus, we know that  $M$ has at least linear area growth.
Hence, for any $k\in D$ and any $t$ satisfying $r\leq t\leq k$,  we have
\begin{equation}
{\rm Area}(B_{t+2}(0)\cap X(M))-{\rm Area}(B_{t-1}(0)\cap X(M))\leq 2C_2(n,\lambda)\varepsilon.
\end{equation}
Since
\begin{equation}
\int_M u^2d\mu\geq {\rm Area}(B_{t+1}(0)\cap X(M))-{\rm Area}(B_{t}(0)\cap X(M)),
\end{equation}
holds, if we choose $\varepsilon$ such that $2C_2(n,\lambda)\varepsilon2^{\frac{n}{2}}<1$,
from \eqref{eq:4-21-5}, we obtain
\begin{equation}
\aligned
&\ \
({\rm Area}(B_{t+1}(0)\cap X(M))-{\rm Area}(B_{t}(0)\cap X(M)))\ln (2^{\frac{n}{2}+1}C_2(n,\lambda)\varepsilon)^{-1}\\
&\leq C_0\biggl({\rm Area}(B_{t+2}(0)\cap X(M))-{\rm Area}(B_{t-1}(0)\cap X(M))\biggl).
\endaligned
\end{equation}
Iterating from $t=r$ to $t=k$ and taking summation on $t$, we infer, from the lemma \ref{lemma:4-20-1}
\begin{equation}
\aligned
&\ \
({\rm Area}(B_{k+1}(0)\cap X(M))-{\rm Area}(B_{r}(0)\cap X(M)))\ln (2^{\frac{n}{2}+1}C_2(n,\lambda)\varepsilon)^{-1}\\
&\leq 3C_0{\rm Area}(B_{k+2}(0)\cap X(M))\leq 6C_0 {\rm Area}(B_{k+1}(0)\cap X(M)).
\endaligned
\end{equation}
Hence, we get
\begin{equation}\label{eq:4-21-6}
\aligned
&{\rm Area}(B_{k+1}(0)\cap X(M))\\
&\leq \frac{\ln (2^{\frac{n}{2}+1}C_2(n,\lambda)\varepsilon)^{-1}}
{\ln (2^{\frac{n}{2}+1}C_2(n,\lambda)\varepsilon)^{-1}-6C_0}{\rm Area}(B_{r}(0)\cap X(M))\\
&\leq \frac{\ln (2^{\frac{n}{2}+1}C_2(n,\lambda)\varepsilon)^{-1}}
{\ln (2^{\frac{n}{2}+1}C_2(n,\lambda)\varepsilon)^{-1}-6C_0}\varepsilon r.
\endaligned
\end{equation}
We can choose $\varepsilon$ small enough such that
\begin{equation}\label{eq:4-21-8}
\frac{\ln (2^{\frac{n}{2}+1}C_2(n,\lambda)\varepsilon)^{-1}}
{\ln (2^{\frac{n}{2}+1}C_2(n,\lambda)\varepsilon)^{-1}-6C_0}\leq 2.
\end{equation}
Therefore,  it follows from \eqref{eq:4-21-6} that
\begin{equation}\label{eq:4-21-7}
{\rm Area}(B_{k+1}(0)\cap X(M))\leq 2\varepsilon r,
\end{equation}
for any $k\in D$. Since $ k+1\geq r$,  we have,  from \eqref{eq:4-21-7} and the definition of $D$, that $k+1\in D$.
Thus,  by induction, we know that $D$ contains all  of  integers $k\geq r$ and
\begin{equation}
{\rm Area}(B_{k}(0)\cap X(M))\leq 2\varepsilon r,
\end{equation}
for any integer $k\geq r$. This implies that $M$ has finite volume, which contradicts with the lemma \ref{lemma:4-21-1}.
Hence, there exist  constants $C$ and $r_0$ such that ${\rm Area}(B_{r}(0)\cap X(M))\geq C r$ for $r>r_0$.
It  completes the proof of the theorem \ref{theorem:4-21-1}.

\end{proof}

\begin{remark}
The estimate in our theorem  is the best possible because the cylinders $S^{n-1}(r_0)\times \mathbb{R}$
satisfy the equality.
\end{remark}

\end {document}